\documentclass{article}

\usepackage{graphicx} 
\usepackage{amsmath}
\numberwithin{equation}{section}
\usepackage{amssymb}
\usepackage{appendix}
\usepackage{amsthm}
\usepackage{xcolor}
\usepackage{booktabs} 
\usepackage{comment}
\usepackage{subcaption}
\usepackage{multirow}
\usepackage{mathtools}
\usepackage[colorlinks,citecolor=blue,urlcolor=blue]{hyperref}

\patchcmd{\thebibliography}
  {\settowidth}
  {\setlength{\itemsep}{-3pt }\settowidth}
  {}{}
\apptocmd{\thebibliography}
  {\small}
  {}{}

\newtheorem{theorem}{Theorem}[section]

\newtheorem{assumption}{Assumption}[section]
\newtheorem{lemma}{Lemma}[section]
\newtheorem{corollary}{Corollary}[section]

\counterwithin{figure}{section}
\counterwithin{table}{section}

\newcommand{\R}{\mathbb{R}}
\newcommand{\E}{\mathbb{E}}

\newcommand{\W}{\mathcal{W}}

\newcommand{\osc}{f_{\text{osc}}}

\newcommand\numberthis{\addtocounter{equation}{1}\tag{\theequation}}
\providecommand{\d}{\textup{d}}
\renewcommand{\d}{d}

\newcommand{\pd}[2]{\frac{\partial #1}{\partial #2}}
\newcommand{\tr}{\text{tr}}

\newcommand{\var}{\mathbb{V}\text{ar}\,}

\newcommand{\argmin}{\text{arg}\,\text{min}}

\newcommand\abs[1]{\left\lvert#1\right\rvert}

\DeclareMathOperator{\diagon}{Diag}

\title{Well-posedness and approximation of reflected McKean-Vlasov SDEs with applications }

\author{P. D. Hinds\thanks{School of Mathematical Sciences, University of Nottingham,  UK; pmxph7@nottingham.ac.uk}, A. Sharma\thanks{Department of Mathematical Sciences, Chalmers University of Technology and the University of Gothenburg, Sweden; akashs@chalmers.se} \, and M. V. Tretyakov\thanks{School of Mathematical Sciences, University of Nottingham,  UK; Michael.Tretyakov@nottingham.ac.uk} }

\begin{document}

\maketitle

\begin{abstract}
In this paper, we establish well-posedness of reflected McKean-Vlasov SDEs and their particle approximations in smooth non-convex domains. We prove convergence of the interacting particle system to the corresponding mean-field limit with the optimal rate of convergence. We motivate this study with applications to sampling and optimization in constrained domains by considering reflected mean-field Langevin SDEs and two reflected consensus-based optimization (CBO) models, respectively. We utilize reflection coupling to study long-time behaviour of reflected mean-field SDEs and also investigate convergence of the reflected CBO models to the global minimum of a constrained optimization problem.  We numerically test reflected CBO models on benchmark constrained optimization problems and an inverse problem.
 
 \medskip
       
 \noindent {\bf Keywords:}  constrained optimization, constrained sampling, mean-field Langevin dynamics, consensus-based optimization, interacting particle systems, reflected mean-field diffusion, reflected stochastic differential equations, propagation of chaos.

 \noindent {\bf AMS Classification:}  60H10 90C26 65C30 65C35 60J76.
     
\end{abstract}

\section{Introduction}

Reflected stochastic differential equations (SDEs) are used to model processes confined to a domain with boundary, where the solution is reflected along a certain direction when it hits the boundary. At the same time, McKean-Vlasov SDEs have coefficients with non-linear dependencies on the law of the solution. This paper is devoted to the study of reflected McKean-Vlasov SDEs which can model systems with constraints and mean-field interactions. We illustrate practical relevance of reflected mean-field SDEs via problems in sampling and optimization with constraints.  

Sampling and optimization problems are ubiquitous in the applied sciences and have found resurgence due to advancements in machine learning. The relation between optimization and sampling is quite established. While convex optimization (as well as log-concave sampling) is well-studied \cite{nocedal_numerical_2006}, the task of non-convex global optimization (as well as non-log-concave sampling) poses additional challenges \cite{horst_introduction_2000}. In the Bayesian setting, many of the optimization problems can be posed as sampling problems. On the other hand, many sampling techniques can be viewed as optimization in infinite dimensional setting through the variational perspective.  With regard to global optimization, \textit{metaheuristic methods} are a popular class of methods which have been used to numerically solve global optimization problems. Such methods consist of a high-level algorithmic framework to coordinate low-level heuristics in order to efficiently explore and exploit the solution space. Examples include particle swarm optimization \cite{kennedy_particle_1995} and differential evolution \cite{storn_differential_1997} among others. While there is typically limited theoretical foundation for such models, these methods have been found to be effective in practice \cite{vos_meta-heuristics_2000}. We are mainly interested in optimization and sampling techniques based on interacting particle systems driven by reflected SDEs. 

Interacting particle system based methods for optimization and sampling have gained traction (see \cite{leimkuhler2018ensemble, carrillo_analytical_2018, Liu_svgd_2016, stuart_eks_2020, garbuno2020affine, reich2015probabilistic, tretyakov_consensus-based_2023}) due to their enhanced capability to explore a non-convex energy landscape. These systems are driven by SDEs and hence the tools from stochastic calculus, especially mean-field theory \cite{sznitman1991topics} and stochastic numerics \cite{milstein_stochastic_2004} are available to establish their convergence. 

Let us discuss the contributions of this paper along with the comparison with existing literature. We will further discuss the literature at appropriate places later in the paper. Here, we give a high-level overview. In the paper, we first establish existence and uniqueness of reflected McKean-Vlasov SDEs and their particle approximations in a non-convex domain setting (Section~\ref{sec:wp}). The well-posedness was shown in a convex domain setting in \cite{adams_large_2022} and in a non-convex setting, but allowing only first order interaction, in \cite{sznitman_nonlinear_1984}. In \cite{wang2023distribution}  well-posedness was also considered in a non-convex domain setting but for McKean-Vlasov SDEs with singular drift and extra conditions on the diffusion coefficient. 
Our next result is convergence of interacting particle system towards its mean-field limit with the optimal rate of convergence (Section~\ref{sec:particles2mf}). The assumptions that we impose for well-posedness and particle convergence are general enough that many models which have been studied in $\R^d$ (see e.g. \cite{carrillo_analytical_2018, stuart_eks_2020, vaes_hoffmann_stuart2022_cbs, ding2024swarm}) can be reformulated in the framework of reflected McKean-Vlasov SDEs to handle constraints.  We provide three illustrations for application of reflected McKean-Vlasov SDEs to sampling and optimization. The first one is reflected mean-field Langevin dynamics (Section~\ref{subsec_mean_field_Lang}) for which we employ reflection coupling to study its long-time behavior and obtain non-asymptotic bounds (Section~\ref{subsec_reflection_coupling_conv}). In the convex domain setting, this study was conducted in \cite{wang2023exponential}. The second class of models we consider is consensus-based optimization (CBO) models (Section~\ref{sec:rcbo}) for which we establish convergence to the global minimum (Section~\ref{sec:globminim}). We add a repelling force in the CBO model to enhance exploration (Section~\ref{subsec_2.2}) and observe its benefits in numerical testing on the Rosenbrock function (Section~\ref{subsec:repelling}).  We test two discretization schemes for reflected SDEs, namely the penalty and projection schemes, to implement the CBO models for a few benchmark constraint optimization problems including an inverse problem (Section~\ref{sec:exp}).

A natural question that arises is why one would formulate optimization and sampling problems with constraints in the setting of reflected SDEs. We first mention other strategies which have been proposed and employed to handle constraints in the literature on interacting particle system driven by SDEs for the purpose of optimization and sampling. In \cite{borghi_constrained_2023}, a CBO model is considered where a penalty function is added to the objective function $f$, which penalizes the objective when points are outside the feasible set $\bar G$ ($f$ is defined on $\R^d$). For a penalization multiplier $\epsilon > 0$, the following unconstrained problem is considered:
\begin{equation}
    \min\limits_{x \in \R^d} P(x; \epsilon),
\end{equation}
where $P(x; \epsilon) = f(x) + \epsilon r(x)$ with $r(x) >0 $ if the feasible set $\bar G$ is violated else it is zero. An Euler scheme is applied to simulate the corresponding particle system. Convergence of the algorithm to the optimal solution is shown, both when the exact penalty parameter $\epsilon$ is known, and also when it is iteratively updated. In \cite{carrillo_consensus-based_2023}, like \cite{borghi_constrained_2023}, the authors also use a penalized objective function presenting several numerical experiments.
In \cite{bae_constrained_2022}, a projection technique is considered for discrete time CBO model driven by common noise.  It is observed in numerical tests on benchmark functions that common noise may result in less exploration producing a sub-standard performance (see \cite{tretyakov_consensus-based_2023}). In the case of ensemble Kalman inversion, which is again an interacting particle based optimization method for inverse problems (see e.g. \cite{iglesias2013ensemble}), in \cite{neil2019box} box constraints are handled using projection scheme and in \cite{weissmann_2023_EKI} a log-barrier penalty is added to deal with inequality constraints defining the convex feasible region. 

We now illustrate with a simple example to highlight why reflected SDEs (i.e., Skorokhod's dynamics) are the natural candidate to handle constraints in models  arising not only in biological and physical sciences but also in models underlying sampling and optimization techniques. 
Suppose we are interested in minimizing a continuously differentiable Lipschitz function $f : \bar G \rightarrow \R$ where $G \subset \R^d$ is a convex domain with boundary $\partial G$, and $x_{\min}=\argmin_{x \in \bar G} f(x)$ is the unique global minimum. In the continuous-time setting, we can employ the following gradient dynamics posed as Skorokhod's problem:
\begin{align*}
d X(t) = -\nabla f (X(t)) dt + d K(t), \quad X(0) \in G,
\end{align*}
where $K(t)$ is a finite variation non-increasing process which increases only when $X(t) \in \partial G$. We call the pair $(X(t), K(t))$ the solution of the Skorokhod problem, whose existence and uniqueness in the convex domain setting is proved by Tanaka \cite{tanaka_stochastic_1979}. The process $K(t)$ can be written as
\begin{align*}
   K(t) = \int_{0}^{t} I_{ \partial G}(X(s)) \nu(X(s)) d |K|(s),
\end{align*}
where $|K|(t)$ is the total variation  of $K(t)$ and $\nu(x)$ belongs to the set of inward normals at $x \in \partial G$. Using the chain rule, we get
\begin{align*}
    d |X(t) - x_{\min}|^2 = - 2\big((X(t) - x_{\min} )\cdot \nabla f(X(t))\big) dt +  2 I_{ \partial G}(X(t))\big((X(t) - x_{\min}) \cdot \nu(X(t))\big) d|K|(t). 
\end{align*}
Note that  $\nu$ is the inward normal of the convex domain $G$, and if we take $f$ to be strongly convex, then we have for some $\kappa >0$
\begin{align*}
\big((X(t) - x_{\min} )\cdot \nabla f(X(t))\big) \geq \kappa |X(t) - x_{\min}|^2 \text{ and }
I_{ \partial G}(X(t))\big((X(t) - x_{\min}) \cdot \nu(X(t))\big) \leq 0. 
\end{align*}
This results in $d |X(t) - x_{\min}|^2 \leq  - 2\kappa |X(t) - x_{\min}|^2 dt$ and hence $|X(t) - x_{\min}|^2 \leq |X(0) - x_{\min}|^2 e^{-2 \kappa t}.$  
It means as $t \rightarrow \infty$, $X(t) \rightarrow x_{\min} $. 

In a similar manner, if $f$ and $G$ are such that
\begin{align}
    (\nabla f (x) \cdot \nu(x)) \leq 0, \;\; x \in \partial G, 
\end{align}
then via chain rule we have $
    d (f (X(t)) - f(x_{\min})) \leq - | \nabla f(X(t))|^2 dt. 
$
If, in addition, we assume that Polyak's inequality holds for $f$ with constant $\eta >0$, then
$ d (f (X(t)) - f(x_{\min})) \leq -\eta (f (X(t)) - f(x_{\min})) dt$ and hence $(f (X(t)) - f(x_{\min})) \leq f (X(0)) - f(x_{\min})) e^{-\eta t}$.

Although above we have taken the example of deterministic convex optimization, it conveys the strategy of using reflected dynamics. 
In contrast to the approaches of penalizing the objective function (see e.g. \cite{carrillo_consensus-based_2023,borghi_constrained_2023,bae_constrained_2022}), here we do not need to modify $f(x)$, noting that modifying $f(x)$ so that it preserves the global optimum within/near the constraint set $\bar G$ is not a trivial task, especially in high dimensions. Moreover, introducing artificial barriers in the objective function typically leads to SDEs with their coefficients taking very large values outside the domain $\bar G$ (effectively, making the SDEs very stiff) which in turn requires use of numerical methods approximating these stiff SDEs with small time steps or of complex nature. This problem does not arise within the reflected SDEs setting considered in this work.

Further, as it has been seen from the above example, we can split our analysis into two steps. In the first step, we establish convergence of continuous-time dynamics to the desired quantity, which can be the global minimum in an optimization scenario or a functional of the desired measure in the case of sampling. The next step is to approximate these continuous-time reflected dynamics, for which a plethora of numerical discretization schemes are available to us. Each of these schemes can be considered as an optimization or sampling technique. For weak-sense approximation (i.e. approximating expectation of function of dynamics with sufficiently large class of functions), we have Lepingle's procedure in the half space setting \cite{lepingle1995euler}, Milstein's change of coordinates scheme \cite{GN96a}, the half-space Euler scheme  \cite{gobet_half_Space_2001}, symmetrized-reflected scheme \cite{BGT04, leimkuhler2023simplerandom}, and specularly reflected scheme \cite{LST2024cld}. All these numerical schemes are analyzed in a non-convex setting except \cite{lepingle1995euler}.   For mean-square (strong) approximation, we have the projection scheme \cite{PET95, SLO01} and the penalty scheme \cite{SLO01}. We mention that projection and penalty schemes have only been analyzed in the convex bounded domain setting, and even then with suboptimal convergence rate (except when the domain is a convex polytope). We note that study of numerical approximation of reflected SDEs is not the aim of this paper, rather we are interested in well-posedness of mean-field SDEs, their particle approximation, and their large time behavior.

In Section~\ref{section_2}, we establish notation and introduce the set-up of reflected non-linear (in the sense of McKean) SDEs and describe applications to sampling and optimization.  We establish well-posedness in Section~\ref{sec:wp} and prove propagation of chaos in the strong sense in Section~\ref{sec:particles2mf}. We devote Section~\ref{sec:longtime} to large time investigation of reflected mean-field  SDE with additive noise and CBO models. The last section (Section~\ref{sec:exp}) contains several numerical experiments to validate the performance of CBO models with convex as well as non-convex constraints.

\section{Reflected McKean-Vlasov SDEs } \label{section_2}

In this section, we first introduce reflected mean-field SDEs and their particle approximation. Then we illustrate their applications in solving constrained optimization and sampling problems. 

Let $\mathcal{P}(\bar G)$  be the space of probability measures on $\bar G$ and $b:\mathbb{R}_{+}\times \bar G \times \mathcal{P}(\bar G) \rightarrow\R^{d}$  and $\sigma:\mathbb{R}_{+}\times\bar G \times \mathcal{P}(\bar G) \rightarrow\R^{d\times d} $. We will impose conditions on these coefficients later.  By $\nu(x)$ denote the unit inward normal at point $x$ belonging to boundary $\partial G$.  Let $(\Omega, {\cal F}, \mathbb P)$ be a complete, sufficiently rich probability space and ${\cal F}_t$, $0 \leq t \leq T$, be a filtration satisfying the usual hypothesis. Let  $(W(t),{\cal F}_t)$  and $(\textbf{W(t)},{\cal F}_t)$ be standard $d$-dimensional and $Nd$-dimensional Wiener processes, respectively, with $\textbf{W(t)}=(W^1(t), \ldots, W^N(t))^{T}$, where $W^i(t)$ are independent standard $d$-dimensional Wiener processes.
We will also use the notation: ${\cal F}_t^W$ is the natural filtration for the Wiener process $W(t)$ and ${\cal F}_t^{\textbf{W}}$ is the natural filtration for the Wiener process $\textbf{W(t)}$.

Consider a non-linear (in the sense of McKean) Markov process evolving on $\bar{G}$ driven by the SDEs
\begin{equation}
X(t) = X(0) + \int_0^t b(s, X(s), \mathcal{L}_{X(s)}) \d s + \int_0^t \sigma(s, X(s), \mathcal{L}_{X(s)}) \d W(s) + \int_0^t \nu(X(s)) I_{\partial G}(X(s)) \d L(s),     \label{eq:MVsde}
\end{equation}
where $\mathcal{L}_{X(s)}$ is the time marginal law of $X(s)$, and $L(s)$ is a scalar non-decreasing process which increases only when $X({s}) \in \partial G$ (see the precise definition in e.g. \cite{lions_stochastic_1984,Ikeda_Watanabe,Freidlin85}):
$L(t) = \int_{0}^{t}I_{\partial G}\big(X(s)\big)dL(s)$ a.s. We also note (see \cite{lions_stochastic_1984}) that the integral form of the local time term of (\ref{eq:MVsde}),
$K(t)=\int_{0}^{t}\nu(X(s))I_{\partial G}\big(X(s)\big)dL(s)$,
is a $d$-dimensional bounded variation process which increases only when $X(s) \in \partial G$.

The first step towards implementation of models based on (\ref{eq:MVsde}) is their particle approximation.  
Let $\delta_x$ be the Dirac measure defined as $\delta_x(A)=I_A(x)$ with $A$ being a measurable set and, for a collection of random variables written in the tuple form as $Z = (Z^1, \ldots, Z^N)^{\top}$, define the empirical measure
\begin{equation}
    \hat{\mu}_{Z} = \frac{1}{N} \sum_{i=1}^N \delta_{Z^{i}}.
    \label{eq:emp_measure}
\end{equation}
The system of $N$ interacting  particles takes the form
\begin{align}
    X^{i,N}(t) = X^{i,N}(0) &+ \int_0^t b\Big(s, X^{i,N}(s), \hat{\mu}_{\textbf{X}^N(s)}\Big) \d s + \int_0^t \sigma\Big(s, X^{i,N}(s), \hat{\mu}_{\textbf{X}^N(s)}\Big) \d W^i(s) \nonumber \\
    &+ \int_0^t \nu(X^{i,N}(s)) I_{\partial G}(X^{i,N}(s)) \d L^{i,N}(s),
    \label{eq:particle_system}
\end{align}
where $\textbf{X}^N(t)=(X^{1,N}(t), \ldots, X^{N,N}(t))^{\top}$ and $L^{i,N}(t)$ is the local time of the $i$-th particle on the boundary $\partial G$. 
In Section~\ref{sec:wp}, we establish well-posedness of (\ref{eq:MVsde}) and (\ref{eq:particle_system}) under some general assumptions on the coefficients and the domain, and in Section~\ref{sec:particles2mf}, we prove convergence of the particle system to its mean-field limit (\ref{eq:MVsde}). In Section~\ref{subsec_reflection_coupling_conv}, we study large-time behaviour of (\ref{eq:MVsde}) with $\sigma \equiv I$ via reflection coupling.

Different choices of $b$ and $\sigma$ in (\ref{eq:MVsde}) can lead to different models for solving constrained optimization and sampling problems. In Section~\ref{subsec_mean_field_Lang}, we present reflected mean-field Langevin SDEs which can be used for sampling from non-linear measure with compact support. In Sections~\ref{subsec_2.1} and~\ref{subsec_2.2}, based on (\ref{eq:MVsde}),  we present  two optimization methods of the CBO-type which we later analyze.
The CBO methods are derivative-free. Although we present mean-field Langevin and CBO-type models as examples of reflected mean-field SDEs, we highlight that 
 one can also write reflected versions of other interacting particle based sampling and optimization models for which our results on well-posedness from Section~\ref{sec:wp} and convergence of the particle approximation with the optimal rate from Section~\ref{sec:particles2mf} hold by verifying the assumptions without difficulty. Ensemble Kalman-Langevin sampler is proposed for sampling in \cite{stuart_eks_2020}.  Its dynamics orchestrate the interaction among particles via ensemble covariance matrix and also utilize gradient information, making it a gradient based method. The model can be used for constrained sampling by formulating it within the reflected SDEs setting and can  also be turned into a constrained optimization model by exploiting the fact that the measure concentrates on the global minimum for lower temperature. Hence, in this case, this model can be written in the form (\ref{eq:particle_system}).
 The same can be said for swarm gradient dynamics  with measure dependent annealing studied in \cite{ding2024swarm} and swarm dynamics of \cite{bolte2024swarm} which can be put in the framework of reflected SDEs for constrained optimization. As already mentioned, since we prove our results of well-posedness of reflected McKean-Vlasov SDEs and interacting particle system in Section~\ref{sec:wp} and propagation of chaos result  in Section~\ref{sec:particles2mf} for general SDEs, these results are applicable to all the above mentioned additional models without any difficulty. These existence and convergence results will also hold true, with minor modifications, for reflected versions of mean-field ODE based models like the Stein variational gradient descent \cite{Liu_svgd_2016} and ensemble Kalman inversion (noise free setting) \cite{reich2015probabilistic, schillings2017analysis} for sampling and optimization, respectively.

\subsection{Reflected mean-field Langevin dynamics}  \label{subsec_mean_field_Lang}

Consider the following mean-field Langevin dynamics with reflection:
\begin{align} \label{mean_field_Lang}
    d X(t) = - \nabla U(X(t)) dt - \int_{\bar{G}}\nabla V(X(t) - y)\mathcal{L}_{X(t))}(dy) + \sigma dW(t) + \nu(X(t)) dL(t), 
\end{align}
where $U :\bar{G } \rightarrow \mathbb{R} $ is an external potential, $ V :\bar{G } \rightarrow \mathbb{R} $ is an interaction potential, and $\sigma >0$. The primary motivation for this model comes from statistical physics. In $\mathbb{R}^d$ its well-posedness and convergence to equilibrium were studied in \cite{malrieu2003convergence, carrillo2006contractions}. 
This object has recently gathered more interest due to its link to the mean-field point of view on neural networks \cite{sirignano_konstantinos2020mean, mei_nguyen2018mean, hu_siska_szpruch_2021mean}. Empirical investigation in  \cite{leimkuhler2021better} suggests that constraint training of neural network can avoid over-fitting and provide stability. In the  mean-field point of view, the problem of training a two-layer wide network with constraints can be posed as the mean-field reflected dynamics (\ref{mean_field_Lang}).

Let us denote by $(V* \mu) (x)$ the convolution $\int_{\bar{G} } V(x-y) \mu(dy) $. The nonlinear measure given by
\begin{align}
    \tilde{\mu} \propto \exp\Big( -\frac{2}{\sigma^{2}} \big( U + V*\tilde{\mu}\big)\Big)   \label{mean_field_Lang_station_measure}
\end{align}
is invariant for the non-linear Markov processes evolving according to the mean-field Langevin dynamics (\ref{mean_field_Lang}) on $\bar{G}$.   This can be confirmed by checking that the density $\rho$  of the measure $\mu$ satisfies the (non-linear) stationary Fokker-Plank equation:
\begin{align}
   -\nabla \cdot (\rho(x) b(x, \rho)) + \frac{\sigma^{2}}{2} \Delta \rho(x) = 0,\quad x \in G 
\end{align}
with boundary condition:
\begin{align}
    \frac{\sigma^{2}}{2}(\nabla \rho(x) \cdot \nu(x)) - (b(x,\rho)\cdot \nu(x))\rho(x) = 0, \quad x \in \partial G,
\end{align}
where $b(x, \rho) = - \nabla U(x) - \int_{\bar{G}} \nabla V(x- y)\rho(y) dy $. Long-time simulation of the non-linear SDEs (\ref{mean_field_Lang}) can produce samples from the implicitly defined measure  (\ref{mean_field_Lang_station_measure}) supported on $\bar{G}$. If the coefficients satisfy Assumption~\ref{as:new} then it ensures the optimal rate of convergence for particle approximation to mean-field limit (see Theorem~\ref{thm:chaos}).  
The result regarding weak convergence towards equilibrium of Section~\ref{subsec_reflection_coupling_conv}  holds for (\ref{mean_field_Lang}) allowing for non-convex potentials $U$ and $V$.

\subsection{Reflected consensus-based models}\label{sec:rcbo}

In this section we consider consensus-based models in the form of reflected SDEs which can be used for solving the constrained optimization problem: 
\begin{equation}
    \min\limits_{x \in \bar G} f(x), \label{eq:obj}
\end{equation}
where $f : \bar{G} \rightarrow [0, \infty)$ is an objective function (which can be non-convex) and $G \subset \mathbb{R}^{d}$ is a bounded domain with a sufficiently smooth boundary $\partial G$. Since $\bar{G}$ is compact, $f$ obtains its infimum and supremum over $\bar{G}$.

\begin{assumption}[Objective function]
    \label{as:objective}
   (i) The objective function satisfies 
    \begin{equation}
    f_{\min} : = \inf_{x \in \bar{G}} f(x) > 0,
\end{equation}
and the minimizer of $f$, denoted as $x_{\min}$ is unique, i.e.,
there is only $x_{\min} \in \bar{G}$ such that
 $   f(x_{\min}) = f_{\min}.$

  (ii)  There exists $L_f > 0$ such that
    \begin{equation}
        \abs{f(x) - f(y)} \leq L_f \abs{x-y},
    \end{equation}
    for all $x, y \in \bar G$.
\end{assumption}

We denote the supremum of $f$ as 
\begin{equation}
    f_{\max} = \sup_{x \in \bar{G}} f(x).
\end{equation}

\subsubsection{Reflected consensus-based optimization}\label{subsec_2.1}

In \cite{pinnau_consensus-based_2017}, the metaheuristic method of consensus-based optimization (CBO) was introduced. In this model, the energy landscape of the objective function is explored by $N$ interacting particles.
Each particle broadcasts its current location to the other particles via an average location, which is weighted according to the energy landscape: particles where the objective function is small are given higher weights than those where the objective function is large. The position of each particle updates continuously so that each particle drifts towards this weighted average, whilst also exploring its current neighbourhood through random perturbations. These dynamics allow for the continuous-time formulation of the CBO model via SDEs. The mathematical framework for CBO has been developed in \cite{carrillo_analytical_2018, carrillo_consensus-based_2021, ha_convergence_2020, tretyakov_consensus-based_2023}. A survey on the topic is available in \cite{totzeck_trends_2021}. 

 Let us present the reflected version of mean-field consensus-based dynamics (see also \cite{Massimo24}): 
\begin{align}\label{eq:mf_cbo}
    \d X(t) &= -\beta (X(t)- \bar{X}(t))\d t + \sigma \diagon(X(t) - \bar{X}(t))\d W(t) + \nu(X(t)) I_{\partial G}\,(X(t))\d L(t),
\end{align}
where the weighted mean $\bar{X}(t)$ (depending on the objective $f$ and some $\alpha > 0$) is
\begin{align}
    \bar{X}(t) :=  \bar{X}(\mathcal{L}_{X(t)}) = \frac{\int_{\bar{G}}x e^{-\alpha f(x)}\mathcal{L}_{X(t)}(\d x)}{\int_{\bar{G}}e^{-\alpha f(x)}\mathcal{L}_{X(t)}(\d x)},
\end{align}
$\beta>0$ and $\sigma>0$ are constant (see a discussion in \cite{tretyakov_consensus-based_2023} on potential benefits of choosing them dependent on time). The relationship between $\beta$ and $\sigma$ required for convergence of $X(t)$ to the global minimum  $x_{\min}$ will be established within convergence analysis in Section~\ref{sec:globminim}.
 
 The particle approximation of (\ref{eq:mf_cbo}) is given by 
\begin{align}\label{particle_cbo}
    \d X^{i, N}(t) &= - \beta (X^{i,N}(t) - \bar{X}^{N}(t)) \d t + \sigma  \diagon(X^{i,N}(t) -  \bar{X}^{N}(t)) \d W^{i}(t) \nonumber  \\ & \;\;\; + \nu(X^{i,N}(t)) I_{\partial G}(X^{i,N}(t))\d L^{i,N}(t),\;\;\; X^{i,N}(0) = X^{i,N}_{0},
\end{align}
where the weighted mean $\bar{X}^{N}(t)$ is 
\begin{align}\label{interac_term}
    \bar{X}^{N}(t) := \bar{X} ^ {N, f, \alpha}(t) := \frac{\sum_{i=1}^{N} X^{i,N}(t) e^{-\alpha f(X^{i,N}(t))}}{\sum_{i=1}^{N}e^{-\alpha f(X^{i,N}(t))}},
\end{align}
which we will refer to as the particles' consensus.  
 In corresponding particles system (\ref{particle_cbo}), the particles interact with each other via (\ref{interac_term}) and try to realize a uniform consensus. 

 Constrained CBO has been considered in \cite{borghi_constrained_2023, bae_constrained_2022, carrillo_consensus-based_2023}. The distinction of our work lies in the fact that we encode the constraints into the continuous-time model with the use of reflected SDEs. While completing this work, we became aware of the paper \cite{Massimo24}, where the model (\ref{eq:mf_cbo}) is also considered. In the convex domain setting, well-posedness, convergence of particle approximation and convergence of CBO model to its global minimum are shown in  \cite{Massimo24} using Tanaka's trick (note that for more general mean-field SDE in convex domains well-posedness and propagation of chaos were earlier proved in \cite{adams_large_2022}). Here, we consider CBO and its new modified version with repelling forces in next subsection as examples and prove well-posedness and convergence of particle system in non-convex domains satisfying uniform exterior sphere condition. For showing convergence to global minimum, we switch to a convex setting and employ techniques introduced in \cite{carrillo_consensus-based_2021}. We discuss in detail the reasoning behind moving to the convex setting in Section~\ref{sec:globminim} and provide useful insight for further development of CBO models.  In addition to the projection numerical method for approximating reflected SDEs (also used in \cite{Massimo24}), we also investigate performance of the penalty scheme for constrained optimization.

\subsubsection{Reflected consensus-based optimization with attracting and repelling forces}\label{subsec_2.2}

In the reflected CBO model (\ref{eq:mf_cbo}), 
the particles are attracted towards its weighted mean portraying exploiting behavior based on already searched space, and the exploration is achieved thanks to noise induced by independent Brownian motions driving each particle.  In \cite{tretyakov_consensus-based_2023}, the exploration is enhanced by adding compound Poisson processes. Another way to facilitate exploration is by incorporating repelling forces among particles with a decaying parameter. The model embodying both attractive and repelling behavior takes the form
\begin{align}\label{eq:mf_cbo_repul}
    \d X(t) &= -\beta (X(t) - \bar{X}(t))\d t + \lambda(t) \int_{\mathbb{R}^{d}}(X(t)- y)\exp{\Big(-\frac{1}{2}| X(t)  - y|^{2}\Big)}\mathcal{L}_{X(t)}(\d y) dt\nonumber \\& \;\;\; + \sigma \diagon(X(t) - \bar{X}(t))\d W(t) + \nu(X(t)) I_{\partial G}\,(X(t))\d L(t),
\end{align}
and its particle approximation is given by
\begin{align}
    \d X^{i,N}(t) & = - \beta (X^{i,N}(t) - \bar{X}^{N}(t)) \d t+ \frac{\lambda(t)}{N} \sum_{j=1}^{N}(X^{i, N}(t) - X^{j,N}(t))\exp\Big( - \frac{1}{2}|X^{i, N}(t) - X^{j,N}(t)|^{2}  \Big)\d t \nonumber \\& \;\;\;  +  \sigma \diagon((X^{i,N}(t) - \bar{X}^{N}(t)))\d W^{i}(t) + \nu(X^{i,N}(t))I_{\partial G}\,(X^{i,N}(t))dL^{i,N}(t).
    \label{CBO_func}
\end{align}
Here $\beta>0$ and $\sigma>0$ are constant, and $\lambda(t) \geq 0$ is a decreasing function of $t$ and preferably should have exponential decay in later steps of the method.
The repelling force between any two particles decays with increase of the distance between the two particles to ensure there is no explosion in the dynamics. When particles are close to each other they experience more repelling to avoid collapse of the ensemble at a local minimum. 
The relation among $\beta$, $\sigma$ and $\lambda$ will be discussed in Section~\ref{sec:globminim}.

\section{Well-posedness results} \label{sec:wp}

In this section, we aim to show: (i) well-posedness of the mean-field SDEs \eqref{eq:MVsde} (existence and uniqueness; Section~\ref{subsec:wp_mvsde}) and (ii) well-posedness of the particle system \eqref{eq:particle_system} (existence and uniqueness; Section~\ref{subsec:wp_particle}). 
We also verify (Section~\ref{subsec:wp_ccbo}) that the assumptions imposed on the coefficients of \eqref{eq:MVsde} and on the domain $G$ under which these well-posedness results hold are satisfied for the two CBO models from Sections~\ref{subsec_2.1} and~\ref{subsec_2.2}.



\subsection{Well-posedness of reflected McKean-Vlasov SDEs}
\label{subsec:wp_mvsde}

If $G$ were a convex domain, then it is easier to deal with the local time term in \eqref{eq:MVsde} using Tanaka's trick \cite{tanaka_stochastic_1979} which is exploited in the McKean-Vlasov setting in \cite{adams_large_2022}. 
Here we do not assume that $G$ is convex, but instead assume that $G$ is bounded and rely on the boundary $\partial G$ satisfying the uniform exterior sphere condition. Sznitman \cite{sznitman_nonlinear_1984} considered this case of non-convex $G$ but only for a first-order interaction. Note that the optimization and sampling models discussed in Section~\ref{section_2}  do not fall into this category, and hence we study the general mean-field SDEs \eqref{eq:MVsde}.

The standard way to prove existence and uniqueness is to appeal to a fixed point argument. To do this, one shows that the map from an arbitrary measure to the measure of the solution of the mean-field SDEs constructed from the arbitrary measure is a contraction. In our setting, this standard fixed-point argument can be used. The existence and uniqueness of a fixed point corresponds to the existence and uniqueness of a solution to the reflected mean-field SDEs \eqref{eq:MVsde}.

For a measurable space $S$, let $\mathcal{P}(S)$ denote the space of probability measures on $S$.
We let $\mathcal{C} = C(\R_+, \bar G)$, equipped with the Borel $\sigma$-algebra generated by the uniform norm topology, and let $\mathcal{M} = \mathcal{P}(\mathcal{C})$.
For a measure $\mu \in \mathcal{M}$, we denote by $\mu_t$ its projection onto $\mathcal{P}(\bar G)$ at time $t$, i.e. its marginal:
\begin{equation}
    \mu_t(A) = \int_\mathcal{C} I_A(w_t) \d \mu(w), \quad A \in \mathcal{B}(\bar G).
\end{equation}
For two measures, $\mu^1, \mu^2 \in \mathcal{P}(\bar G)$, let $\W_p$ be the Wasserstein $p$-metric defined as
\begin{equation}
    \W_p(\mu^1, {\mu}^2) := \inf_{\gamma \in \Gamma(\mu^1, \mu^2)} \bigg[ \int_{\bar G \times \bar G} \abs{x-y}^p \d \gamma(x, y) \bigg]^{\frac{1}{p}},
    \label{eq:d_metric}
\end{equation}
where $\Gamma(\mu^1, {\mu}^2)$ denotes the set of couplings between $\mu^1$ and ${\mu}^2$.


\begin{assumption}
    \label{as:smooth}
    The domain $G \subset \R^d$ is bounded and its boundary $\partial G$ is $C^3$.
\end{assumption}
The smoothness of the boundary assumed here is required so that the distance function to the boundary of $G$, $d(\cdot, \partial G)$, defined on a neighbourhood of $\partial G$ is $C^2$ (see \cite[Section~3, Lemma~1]{serrin_problem_1969}). We make use of this in Theorem~\ref{thm:wp_mf}.

\begin{assumption}[Uniformly Lipschitz in space and measure]
    \label{as:lipschitz}
     There exists $L > 0$ such that for all $t \geq 0$
     \begin{align}
         \abs{b(t, x, \mu^1) - b(t, y, \mu^2)} + \abs{\sigma(t, x, \mu^1) - \sigma(t, y, \mu^2)} \leq L (\abs{x - y} + \W_4(\mu^1, \mu^2)), \, \forall x, y \in \bar G, \, \forall \mu^1, {\mu}^2 \in \mathcal{P}(\bar G).
     \end{align}
    Also, the coefficients $b(t, x, \mu)$ and $\sigma(t, x, \mu)$ are continuous in $t$. 
\end{assumption}

Since $\partial G$ is sufficiently smooth, it satisfies the uniform exterior sphere condition. That is, there exists $R_0 > 0$ such that $\forall x \in \partial G$,
\begin{equation*}
    \bar B(x - R_0 \nu(x), R_0 ) \cap \partial G = \{ x \},
\end{equation*}
where $\bar B(x, R_0)$ denotes the closed ball of radius $R_0$ centred at $x$ and $\nu(x)$ is the inward normal vector field at $x \in \partial G$.
The constant $r$ is called the uniform exterior sphere constant of $G$. This is equivalent to the following condition: 
\begin{equation}
    \exists c > 0, \quad \forall x \in \partial G, \quad \forall y \in \bar{G}, \quad c \abs{x-y}^2 \geq (x-y)^\top \nu (x),
    \label{eq:ext_sphere_condition}
\end{equation}
see for example \cite{pilipenko_introduction_2014}.
The constant $c$ is related to the uniform exterior sphere constant by $c = \frac{1}{2R_0}$.


\begin{theorem}
    \label{thm:wp_mf}
    Let Assumptions~\ref{as:smooth} and \ref{as:lipschitz} hold. There exists a unique pair of continuous ${\cal F}_t^W$-adapted processes $(X(t), L(t))$ such that (i) $X(t) \in \bar{G}$ for all $t \geq 0$; (ii) $L(t)$ is non-decreasing with $L(0) = 0$ and for all $t\geq 0$,
    \begin{equation}
                L(t) = \int_0^t I_{\partial G}(X(s)) \d L(s); 
            \end{equation}
    and (iii)  for all $t \geq 0$, 
            \begin{equation}
                X(t) = X(0) + \int_0^t b(s, X(s), \mathcal{L}_{X(s)}) \d s + \int_0^t \sigma(s, X(s), \mathcal{L}_{X(s)}) \d W(s) + \int_0^t \nu(X(s)) \d L(s).
                \label{eq:MVsde_3.8}
            \end{equation}       
\end{theorem}

Such a theorem is proved in \cite{sznitman_nonlinear_1984} in the case of first-order interaction rather than $\eqref{eq:MVsde}$.
Our proof of Theorem~\ref{thm:wp_mf} borrows certain arguments and notation from \cite{sznitman_nonlinear_1984} which also draws techniques from \cite{lions_stochastic_1984} to deal with the local time term.

\begin{proof}[Proof of Theorem \ref{thm:wp_mf}]
For brevity, we denote $X_s := X(s)$, 
$W_s := W(s)$, and $L_s : = L(s)$. 

    Consider the map $\Phi : \mathcal{M} \rightarrow \mathcal{M}$ defined by $\Phi(\mu) = \mathcal{L}_{Y^\mu}$ where $\mathcal{L}_{Y^\mu}$, is the law of $\{Y^\mu_t\}_{t \geq 0}$ which is defined as
    \begin{equation}
        Y^\mu_t = X_0 + \int_0^t b(s, Y^\mu_s, \mu_s) \d s + \int_0^t \sigma(s, Y^\mu_s, \mu_s) \d W(s) 
        + \int_0^t \nu(Y^\mu_s) I_{\partial G}\,(Y^\mu_t) \d L(s).
        \label{eq:Y_def}
    \end{equation}
    The process $Y$ is decoupled from its own measure and, as such, is just an ordinary reflected SDE. The existence and uniqueness of a solution to \eqref{eq:Y_def} is proved in \cite{lions_stochastic_1984} when the coefficients $b$ and $\sigma$ are defined on the whole space $\R^d$. We note that we can extend our coefficients smoothly to the whole space, and the choice of extension does not affect the solution $Y^\mu$. Hence, the map $\Phi$ is well-defined.
    From the definition of $\Phi$, if $\Phi$ has a unique fixed point, then the processes associated with this fixed point are the unique solution to \eqref{eq:MVsde}. 
    Our objective is to show that there exists a unique fixed point by demonstrating that, for some $j$, the $j$-fold composition, $\Phi^j$, is a contraction.
    
    Let us fix an arbitrary $T > 0$ and consider the space $\mathcal{C}_T = C([0, T], \bar G)$ and $\mathcal{M}_T : = \mathcal{P}(\mathcal{C}_T)$ equipped with the metric
    \begin{equation*}
        D_T(\mu^1, \mu^2) = \inf_{\gamma \in \Gamma(\mu^1, \mu^2)} \bigg[ \int_{\mathcal{C}_T \times \mathcal{C}_T} \Big( 
        \sup_{s \leq T}\abs{X(s)-Y(s)} \Big)^4 \d \gamma(X, Y) \bigg]^{\frac{1}{4}}.
    \end{equation*}
    Note that $D_T$ is the Wasserstein $4$-distance with respect to the norm $\sup_{s\leq T}\abs{X(s)}$, $X \in \mathcal{C}_T$.
    Completeness of the space $(\mathcal{M}_T, D_T)$ follows from separability and completeness of $(\mathcal{C}_T, \sup_{s\leq T} \abs{X(s) - Y(s)})$ \cite{bogachev_monge-kantorovich_2012, bolley_separability_2008}. 
    
    \begin{lemma}
        Under Assumptions~\ref{as:smooth} and \ref{as:lipschitz}, there exists $K_T > 0$ such that for all $\mu_1, \mu_2 \in \mathcal{M}$,
        \begin{equation}
            D_T^4 \big(\Phi(\mu^1), \Phi(\mu^2)\big) \leq K_T \int_0^T D_u^4(\mu^1, \mu^2) \d u.
            \label{eq:contraction}
        \end{equation}
        \label{lemma:contraction}
    \end{lemma}
    Once the lemma is proved, we can apply \eqref{eq:contraction} twice to yield
    \begin{align}
        D^4_T\big( \Phi (\Phi(\mu^1)), \Phi(\Phi(\mu^2))  \big) \leq K_T \int_{0}^T D^4_{t_1} \big( \Phi(\mu^1), \Phi(\mu^2) \big) \d t_1 \leq K_T^2 \int_0^T \int_0^{t_1} D_{t_2}(\mu^1, \mu^2) \d t_2 \d t_1.
    \end{align}
    Continuing in this way for the $j$-fold composition $\Phi^j$, we obtain
    \begin{align}
        D^4_T\big( \Phi^j (\mu^1), \Phi^j(\mu^2)  \big) &\leq K_T^j \int_{0}^T \int_0^{t_1} \cdots \int_{0}^{t_{j-1}} D^4_{t_j}(\mu^1, \mu^2) \d t_j \cdots \d t_1.
    \end{align}
    Changing the order of integration yields
    \begin{align}
        D^4_T\big( \Phi^j (\mu^1), \Phi^j(\mu^2)  \big) &\leq K_T^j \int_0^T \int_{t_j}^T \cdots \int_{t_2}^T D^4_{t_j}(\mu^1, \mu^2) \d t_1 \cdots \d t_j \\
        &= K_T^j \int_0^T \frac{(T - t_j)^{j-1}}{(j-1)!} D_{t_j}^4 (\mu^1, \mu^2) \d t_j
        \leq \frac{(K_T T)^j}{j!} D_T^4 (\mu^1, \mu^2).
        \label{eq:contraction2}
    \end{align}
    Hence, if we choose $j$ large enough such that
    \begin{equation}
        \frac{(K_T T)^j}{j!} < 1, \nonumber
    \end{equation}
    it follows that $\Phi^j$ is a contraction and $\Phi$ has a unique fixed point, as required to complete the proof. 
 \end{proof}
 
    We now prove Lemma~\ref{lemma:contraction}.


    \begin{proof}[Proof of Lemma \ref{lemma:contraction}]
        Let $\mu^1, \mu^2 \in \mathcal{M}$ and define the processes $Y^1_t := Y^{\mu^1}_t, Y^2_t := Y^{\mu^2}_t$ via \eqref{eq:Y_def}.
        Let $d(\cdot, \partial G)$ denote the distance to the boundary of $G$, defined on some 
        neighborhood of $\partial G$ within $G$ and let the function $g(\cdot)$ be a smooth, bounded extension of it to the whole space.
        Then Ito's formula gives
        \begin{align}
            g(Y^i_t) = g(X_0) &+ \int_0^t \nabla g(Y^i_s) ^\top b(s, Y^i_s, \mu^i_s) + \frac{1}{2}\text{tr}[\sigma(s, Y^i_s, \mu^i_s)^\top \nabla^2 g(Y^i_s) \sigma(s, Y^i_s, \mu^i_s)] \d s  \\
            &+ \int_0^t \nabla g(Y^i_s) ^\top \sigma(s, Y^i_s, \mu^i_s) \d W_s + \int_0^t \nabla g^\top (Y^i_s) \nu (Y^i_s) \d L_s
        \end{align}
        for $i=1, 2$, where $\nabla^2$ denotes the Hessian. Noting that $\nabla g = \nu$ on $\partial G$, we have
        \begin{equation}
            \int_0^t \nabla g^\top (Y^i_s) \nu (Y^i_s) \d L_s = L_t.
        \end{equation}
        Hence,
        \begin{align}
            g(Y^i_t) = g(X_0) &+ \int_0^t \Big( \nabla g(Y^i_s) ^\top b(s, Y^i_s, \mu^i_s) + \frac{1}{2}\text{tr}[\sigma(s, Y^i_s, \mu^i_s)^\top \nabla^2 g(Y^i_s) \sigma(s, Y^i_s, \mu^i_s)] \Big)\d s  \\
            &+ \int_0^t \nabla g(Y^i_s) ^\top \sigma(s, Y^i_s, \mu^i_s) \d W_s + L_t, 
        \end{align}
        In the interest of brevity, we write this as
        \begin{equation}
            g^i_t = g_0 + \int_0^t \Big( {\nabla g^i_s} ^\top b^i_s + \frac{1}{2}\text{tr}[{\sigma^i_s}^\top \nabla^2 g^i_s \sigma^i_s] \Big) \d s + \int_0^t {\nabla g^i_s} ^\top \sigma^i_s \d W_s + L_t^i, \,\, i=1, 2.
        \end{equation}
       
        Let $c > 0$ be the uniform exterior sphere constant from \eqref{eq:ext_sphere_condition}.  By Ito's formula, we have
            \begin{align}
                \exp\{-2c(g^1_t + g^2_t)\} &= \exp \{ - 4 c g_0\} + \int_0^t \exp\{-2c(g^1_s + g^2_s)\} \Big(-2c \tilde{b}_s + 2c^2 \tilde{\sigma}_{s}\tilde{\sigma}_{s}^{\top} \Big)  \d s  \nonumber\\
                &- \int_0^t 2c \exp\{-2c(g^1_s + g^2_s)\} \tilde\sigma \d W_s 
                - \int_0^t 2c \exp\{-2c(g^1_s + g^2_s)\} [\d L_s^1 + \d L_s^2],
            \end{align}
            where 
            \begin{align}
                \tilde{b}_s = (\nabla {g^1_s})^\top b^1_s + (\nabla {g^2_s})^\top b^2_s + \frac{1}{2}\text{tr}[({\sigma^1_s})^\top \nabla^2 g^1_s \sigma^1_s + ({\sigma^2_s})^\top \nabla^2 g^2_s \sigma^2_s],  \,\,       
                \tilde{\sigma}_s = (\nabla {g^1_s})^\top \sigma^1_s + (\nabla {g^2_s})^\top \sigma^2_s.
            \end{align}
            Also,
            \begin{align}
                \abs{Y^1_t - Y^2_t}^2 &= \int_0^t 2 (Y^1_s - Y^2_s)^\top (b^1_s - b^2_s) + \tr [(\sigma^1_s - \sigma^2_s)^\top (\sigma^1_s - \sigma^2_s)] \d s \nonumber \\
                &\qquad+ \int_0^t 2 (Y^1_s - Y^2_s)^\top (\sigma^1_s - \sigma^2_s) \d W_s 
                + \int_0^t 2 (Y^1_s - Y^2_s)^\top [\nu(Y^1_s) \d L^1_s - \nu(Y^2_s) \d L^2_s].
            \end{align}
        
        Then, for the product, Ito's formula yields
        \begin{align}
            \exp & \{-2c (g^1_t + g^2_t)\} \abs{Y^1_t - Y^2_t}^2 \nonumber \\
                & = 2  \int_0^t \exp\{-2c(g^1_s + g^2_s)\}(Y^1_s - Y^2_s)^\top [(b^1_s - b^2_s)\d s + (\sigma^1_s - \sigma^2_s)\d W_s + \nu(Y^1_s)\d L^1_s \nonumber\\
                & \qquad- \nu(Y^2_s) \d L^2_s] + \int_0^t \exp\{-2c(g^1_s + g^2_s)\} \text{tr}[(\sigma_s^1 - \sigma_s^2)^\top (\sigma_s^1 - \sigma_s^2)] \d s \nonumber\\
                & \qquad- 2c\int_0^t \abs{Y^1_s - Y^2_s}^2 \exp\{-2c(g^1_s + g^2_s)\}[\tilde{b}_s \d s + \Tilde{\sigma}^\top_s \d W_s + \d L^1_s + \d L^2_s] \nonumber\\
                & \qquad + 2c^2\int_0^t  \exp\{-2c(g^1_s + g^2_s)\} \abs{Y^1_s - Y^2_s}^2 \tilde{\sigma}_s \tilde{\sigma}^\top_s \d s \nonumber\\
                & \qquad - 4c\int_0^t \exp\{-2c(g^1_s + g^2_s)\} (Y(s)^1 - Y^2_s)^\top (\sigma_s^1 - \sigma_s^2) \tilde{\sigma}^{\top} \d s.
        \end{align}
        
        From the uniform exterior sphere condition \eqref{eq:ext_sphere_condition}, we have
        \begin{align}
            - c \abs{Y^1_s - Y^2_s}^2 + (Y^1_s - Y^2_s)^ \top \nu(Y^1_s) \leq 0, 
            - c \abs{Y^1_s - Y^2_s}^2 + (Y^2_s - Y^1_s) ^\top \nu(Y^2_s) \leq 0, \,\, a.s.
        \end{align}
        To ease notation, let us denote $\kappa(s) : = \exp\{-2c(g^1_s + g^2_s)\}$, and note that $\kappa$ is bounded since $g$ is bounded. Then,
        \begin{align}
            \kappa(t) \abs{Y^1_t - Y^2_t}^2 & \leq 2  \int_0^t \kappa(s)(Y^1_s - Y^2_s)^\top [(b^1_s - b^2_s)\d s + (\sigma^1_s - \sigma^2_s)\d W_s] \nonumber\\
                & \qquad + \int_0^t \kappa(s) \text{tr}[(\sigma_s^1 - \sigma_s^2)^\top (\sigma_s^1 - \sigma_s^2)] \d s 
               - 2c\int_0^t \abs{Y^1_s - Y^2_s}^2 \kappa(s)[\tilde{b}_s \d s + \Tilde{\sigma}_s^\top \d W_s] \nonumber\\
                & \qquad + 2c^2\int_0^t  \kappa(s) \abs{Y^1_s - Y^2_s}^2 \tilde{\sigma}_s \tilde{\sigma}^\top_s \d s 
                 - 4c\int_0^t \kappa(s) (Y(s)^1 - Y^2_s)^\top (\sigma_s^1 - \sigma_s^2) \tilde{\sigma}^\top_s \d s.
        \end{align}
        Squaring both sides and applying the Cauchy-Schwarz inequality twice, we obtain for some constant $K>0$
        \begin{align}
            \abs{Y^1_t - Y^2_t}^4 & \leq K \Bigg [ \int_0^t \big[(Y^1_s - Y^2_s)^\top (b^1_s - b^2_s)\big]^2\d s + \int_0^t \big[\text{tr}[(\sigma_s^1 - \sigma_s^2)^\top (\sigma_s^1 - \sigma_s^2)]\big]^2 \d s \nonumber\\
                & \qquad+ \int_0^t \abs{Y^1_s - Y^2_s}^4 \d s + \int_0^t \big[(Y(s)^1 - Y^2_s)^\top (\sigma_s^1 - \sigma_s^2) \tilde{\sigma}_s\big]^2 \d s \nonumber\\
                & \qquad + \bigg(\int_0^t \kappa(s) (Y^1_s - Y^2_s)^\top (\sigma_s^1 - \sigma_s^2) \d W_s\bigg)^2 
                + \bigg(\int_0^t \kappa(s) \abs{Y^1_s - Y^2_s}^2 \tilde{\sigma}_s \d W_s\bigg)^2 \Bigg],
        \end{align}
        where we use the fact that $\kappa(s)$ is bounded by definition, and $\tilde{b}$ and $\tilde{\sigma}$ are bounded on $\bar G$.
                Then, taking the supremum over $[0, t]$ and taking expectation, we get
        \begin{align}
            \E&\sup_{u \leq t} \abs{Y^1_u - Y^2_u}^4  \leq K \Bigg [ \E \sup_{u \leq t} \bigg ( \int_0^u \big[(Y^1_s - Y^2_s)^\top (b^1_s - b^2_s)\big]^2\d s + \int_0^u \big[\text{tr}[(\sigma_s^1 - \sigma_s^2)^\top (\sigma_s^1 - \sigma_s^2)]\big]^2 \d s \nonumber\\
                & + \int_0^u \abs{Y^1_s - Y^2_s}^4 \d s + \int_0^u \big[(Y_s^1 - Y^2_s)^\top (\sigma_s^1 - \sigma_s^2) \tilde{\sigma}_s\big]^2 \d s \bigg) \nonumber\\
                &  + \E \sup_{u\leq t}\bigg(\int_0^u \kappa(s) (Y^1_s - Y^2_s)^\top (\sigma_s^1 - \sigma_s^2) \d W_s\bigg)^2 
               + \E \sup_{u\leq t}\bigg(\int_0^u \kappa(s) \abs{Y^1_s - Y^2_s}^2 \tilde{\sigma}_s \d W_s\bigg)^2 \Bigg].
        \end{align}
        The Riemann integrals have non-negative integrands so their supremum is reached at $t$. Meanwhile, we use Doob's maximal inequality followed by the Ito isometry for the Ito integrals, which yields
        \begin{align}
            \E\sup_{u \leq t} \abs{Y^1_u - Y^2_u}^4 & \leq K \E \bigg ( \int_0^t \big[(Y^1_s - Y^2_s)^\top (b^1_s - b^2_s)\big]^2\d s + \int_0^t \big[\text{tr}[(\sigma_s^1 - \sigma_s^2)^\top (\sigma_s^1 - \sigma_s^2)]\big]^2 \d s \nonumber\\
                & \qquad+ \int_0^t \abs{Y^1_s - Y^2_s}^4 \d s + \int_0^t \big[(Y(s)^1 - Y^2_s)^\top (\sigma_s^1 - \sigma_s^2) \tilde{\sigma}_s\big]^2 \d s \nonumber\\
                & \qquad + \int_0^t \E \big[(Y^1_s - Y^2_s)^\top (\sigma_s^1 - \sigma_s^2)\big]^2 \d s + \int_0^t \E \abs{Y^1_s - Y^2_s}^4 \d s \bigg),
        \end{align}
        where again we have used boundedness of $\kappa(s)$ and $\tilde{\sigma}$.
        Now using the Cauchy-Schwarz inequality and the Assumption~\ref{as:lipschitz}, we obtain
        \begin{align}
            \E\sup_{u \leq t} \abs{Y^1_u - Y^2_u}^4 & \leq K \E \bigg ( \int_0^t \abs{Y_s^1 - Y^2_s}^2 \abs{b^1_s - b^2_s}^2 + \abs{Y_s^1 - Y^2_s}^2 \abs{\sigma^1_s - \sigma^2_s}^2 \nonumber\\
            & \qquad + \abs{Y_s^1 - Y^2_s}^4 + \abs{\sigma^1_s - \sigma^2_s}^4 \d s \bigg) \label{eq:process_bound0} \\
            &\leq  K \bigg ( \int_0^t \E \abs{Y_s^1 - Y^2_s}^4 + \W_4^4(\mu_s^1, \mu_s^2) \d s \bigg).
            \label{eq:process_bound}
        \end{align}
        {Note that $(Y^1_s - Y^2_s)^\top(\sigma^1_s - \sigma^2_s)\tilde{\sigma} \leq \abs{Y^1_s - Y^2_s}\abs{\sigma_s^1 - \sigma_s^2}$ follows from application of the Cauchy-Schwarz inequality, the operator norm being bounded by Frobenius norm, and $\tilde{\sigma}$ being bounded.} Then,
        \begin{equation}
            \E\sup_{u \leq t} \abs{Y^1_u - Y^2_u}^4 \leq \int_0^t K \E \sup_{u\leq s}\abs{Y_u^1 - Y^2_u}^4 \d s + K \int_0^t \W_4^4(\mu_s^1, \mu_s^2) \d s,
        \end{equation}
        to which Gr\"{o}nwall's lemma can be applied to yield
        \begin{align}
            \E\sup_{u \leq t} \abs{Y^1_u - Y^2_u}^4 & \leq K e^{Kt} \int_0^t \W_4^4(\mu_s^1, \mu_s^2) \d s.
        \end{align}
        Hence
        \begin{align}
            D_T^4 \big( \Phi(\mu^1), \Phi(\mu^2) \big) &\leq \E \sup_{u \leq T} \abs{Y^1_u - Y^2_u}^4 \label{eq:l1_final1} \\
                &\leq K e^{KT} \int_0^T \W_4^4 ( \mu^1_u , \mu^2_u ) \d u \\
                &\leq K e^{KT} \int_0^T D_u^4 (\mu^1, \mu^2) \d u. \label{eq:l1_final3}
        \end{align}
        The inequality \eqref{eq:l1_final1} follows from the definition of the metric $D_T$, while \eqref{eq:l1_final3} is a consequence of $\W_4^4 ( \mu^1_u , \mu^2_u ) \leq D_u^4 (\mu^1, \mu^2)$. Lemma~\ref{lemma:contraction} is proved.
    \end{proof}

\subsection{Well-posedness of the particle system}\label{subsec:wp_particle}

We show well-posedness of the particle system \eqref{eq:particle_system}.
\begin{theorem}
    \label{thm:wp_pa}
    Let Assumptions~\ref{as:smooth} and \ref{as:lipschitz} hold. Then there exists a unique strong solution of the SDEs \eqref{eq:particle_system}.
\end{theorem}

\begin{proof}
For the sake of simplicity in presentation, we denote $X^{i,N}_s := X^{i,N}(s)$, 
$W^{i}_s := W^{i}(s)$ and $L^i_s := L^i(s)$.
Let $\mathbb{S}$ be the space of continuous and ${\cal F}_t^{\textbf{W}}$-adapted $\bar{G}^N$-valued processes. We define the map
\begin{align}
    \Phi &: \mathbb{S} \rightarrow \mathbb{S}, \\
    &Y \mapsto \bar{Y}, \nonumber
\end{align}
where $\bar{Y} = \big( \bar Y^{{1,N}^\top}, \ldots, \bar Y^{{N,N}^\top} \big)^\top$ is the solution to
\begin{align}
    \bar Y^{i,N}_t = X^{i,N}_0 + \int_0^t b(s, \bar Y^{i,N}_s, {\hat\mu}_{Y_s}) \d t + \int_0^t \sigma(s, \bar Y^{i,N}_s, {\hat\mu}_{Y_s}) \d W^i_s + \int_0^t \nu(\bar Y^{i,N}_s) \d L^{i,N}_s, \,\, i = 1, \ldots, N,
    \label{eq:SDE_ybar}
\end{align}
where $\hat\mu_{Y_t}$ is the empirical measure of an input process $Y \in \mathbb{S}$ at time $t$ (see (\ref{eq:emp_measure})).
The SDE \eqref{eq:SDE_ybar} is a reflected SDE with random coefficients. Existence and uniqueness of each component $\bar Y^i$ follow from \cite[Theorem 2.2]{bouchard_optimal_2008}. Hence, the map $\Phi$ is well-defined.

Let $Y^{(1)}, Y^{(2)} \in \mathbb{S}^N$ and let $\bar Y^{(j)} = \Phi(Y^{(j)})$, $j=1, 2$.
Applying the technique of Lemma~\ref{lemma:contraction} to each of the $N$ components of the two processes $\bar Y^{(j)}$, $j=1,2$, we obtain (cf. \eqref{eq:process_bound})
\begin{align}
    \E\sup_{u \leq t} \abs{\bar Y^{(1), i,N}_u - \bar Y^{(2), i,N}_u}^4 \leq 
    K \bigg ( \int_0^t \E \abs{\bar Y^{(1), i,N}_u - \bar Y^{(2), i,N}_u}^4 + \E \W_4^4\Big(\hat \mu_{Y^{(1)}_s}, \hat \mu_{Y^{(2)}_s}\Big) \d s \bigg),
\end{align}
for $i = 1, \ldots, N$. Applying Gr\"{o}nwall's lemma yields
\begin{equation}
    \E\sup_{u \leq t} \abs{\bar Y^{(1), i,N}_u - \bar Y^{(2), i,N}_u}^4 \leq 
    Ke^{Kt} \bigg ( \int_0^t \E \W_4^4\Big(\hat \mu_{Y^{(1)}_s}, \hat \mu_{Y^{(2)}_s}\Big) \d s \bigg).
    \label{eq:particle_bound}
\end{equation}
From the definition \eqref{eq:d_metric} of the Wasserstein distance and from \eqref{eq:emp_measure}, it can be seen that
\begin{equation}
    \E \W_4^4\Big(\hat \mu_{Y^{(1)}_s}, \hat \mu_{Y^{(2)}_s}\Big) \leq \frac{1}{N}  \sum_{i=1}^N\E \abs{Y_s^{(1), i,N} - Y^{(2), i,N}_s}^4.
    \label{eq:w_bound}
\end{equation}
Hence,
\begin{align}
    \E \sup_{u\leq t}\abs{\bar Y_u^{(1)} - \bar Y_u^{(2)}}^4 &\leq K \sum_{i=1}^N \E\sup_{u \leq t} \abs{\bar Y^{(1), i,N}_u - \bar Y^{(2), i,N}_u}^4 \\
    & \leq Ke^{Kt} \int_0^t \sum_{i=1}^N \E \abs{Y_s^{(1), i,N} - Y^{(2), i, N}_s}^4 \d s \label{eq:x1} \\
    & \leq Ke^{Kt} \int_0^t \E \sup_{u \leq s}\abs{Y_u^{(1)} - Y^{(2)}_u}^4 \d s,
\end{align}
where for the second inequality we use \eqref{eq:particle_bound} and then \eqref{eq:w_bound}. One can then show that the map $\Phi$ is a contraction as before, see \eqref{eq:contraction2}. The unique fixed-point of $\Phi$ is the solution to the particle system.
\end{proof}

\subsection{Well-posedness of the CBO models}
\label{subsec:wp_ccbo}

We now verify that the CBO mean-field SDEs and their particle approximations from Sections~\ref{subsec_2.1} and~\ref{subsec_2.2} satisfy Assumption~\ref{as:lipschitz} so that their well-posedness is established. 

For $\mu \in \mathcal{P}(\bar G)$, let us define
\begin{equation}
    \bar{X}^{\mu} := \frac{\int_{\bar G} x e^{-\alpha f(x)} \mu (\d x)}{\int_{\bar G} e^{-\alpha f(x)} \mu(\d x)}.
\end{equation}

The following lemma is proved in \cite{carrillo_analytical_2018}. 
\begin{lemma}
    \label{lm:carrillo}
    Let Assumption~\ref{as:smooth} hold and $f$ satisfy Assumption~\ref{as:objective}. There exists $K > 0$ such that for all $\mu, \bar\mu \in \mathcal{P}(\bar G)$,
    \begin{equation}
        \abs{\bar{X}^{\mu} - \bar{X}^{\bar\mu}} \leq K \mathcal{W}_2(\mu, \bar\mu).
    \end{equation}
  \end{lemma}

Now we show that the coefficients of the CBO model (\ref{eq:mf_cbo}) satisfy Assumption~\ref{as:objective}.
\begin{lemma}
    The coefficients of (\ref{eq:mf_cbo}),
    \begin{align}
        b(t, x, \mu) = \beta  (x - \bar{X}^{\mu}) \,\, \text{and }
        \sigma (t, x, \mu) = \sigma \textup{Diag}(x - \bar{X}^{\mu}),
    \end{align}
    are Lipschitz, uniformly in time, with respect to the space and measure arguments (i.e., they
    satisfy Assumption~\ref{as:lipschitz}).
   \end{lemma} 
    \begin{proof}
        For $x, \bar{x} \in \R^d$ and $\mu, \bar\mu \in \mathcal{P}(\bar{G})$,
        \begin{align}
            &\abs{b(t, x, \mu) - b(t, \bar{x}, \bar\mu)} + \abs{\sigma (t, x, \mu) - \sigma (t, \bar{x}, \bar\mu)} = \beta \abs{(x - \bar{x}) + (\bar{X}^{\bar{\mu}} - \bar{X}^{\mu})} + \sigma \abs{(x - \bar{x}) + (\bar{X}^{\bar{\mu}} - \bar{X}^{\mu})} \nonumber \\
            &\qquad \leq K \big( \abs{x - \bar{x}} + \abs{\bar{X}^{\mu} - \bar{X}^{\bar{\mu}}}\big) 
            \leq K \big ( \abs{x - \bar{x}} + \mathcal{W}_2 (\mu, \bar\mu)  \big) 
             \leq K \big( \abs{x - \bar{x}} + \mathcal{W}_4(\mu, \bar\mu) \big), \nonumber
        \end{align}
        by application of Lemma~\ref{lm:carrillo}.
    \end{proof}
 
Hence, by Theorem~\ref{thm:wp_mf} we have well-posedness of the mean-field limit \eqref{eq:mf_cbo} and the particle system \eqref{particle_cbo}.
In the same way it can be verified that the coefficients of the CBO model (\ref{eq:mf_cbo_repul}) satisfy Assumption~\ref{as:lipschitz}, from which
well-posedness of the mean-field limit \eqref{eq:mf_cbo_repul} and the particle system \eqref{CBO_func} follows.

\section{Convergence of interacting particle system to the mean-field limit} \label{sec:particles2mf}

In this section we prove a theorem on propagation of chaos with the optimal convergence rate. Let $\textbf{X(t)}:=(X^1(t),\ldots,X^N (t))^{\top}$, where $X^i(t)$, $i=1,\ldots,N$, are i.i.d. particles solving the McKean-Vlasov SDEs (\ref{eq:MVsde}) which are driven by the independent Wiener processes $W^i$ that are the same as in the particle system \eqref{eq:particle_system}. 

Under Assumption~\ref{as:smooth} and~\ref{as:lipschitz}, it is not difficult to prove using an analogue of \eqref{eq:process_bound} and the well-known rate of convergence for empirical measures \cite[Theorem 1]{fournier_rate_2015} that for some constant $C>0$: 
\begin{equation}
      \max_{i=1,\ldots,N}  \sup_{t \leq T} \E \abs{X^{i, N}(t) - X^i (t)}^4\leq C
        \begin{cases}
            N^{-1/2}, \quad &d < 8, \\
            N^{-1/2} \log N, \quad &d = 8, \\
            N^{-\frac{4}{d}}, \quad &d > 8.
        \end{cases}
\end{equation}

Under the below Assumption~\ref{as:new} replacing Assumption~\ref{as:lipschitz}, we instead prove the optimal convergence rate of order $1/\sqrt{N}$. The CBO models from Sections~\ref{subsec_2.1} and \ref{subsec_2.2} satisfy Assumption~\ref{as:new}.

\begin{assumption}\label{as:new}
There exist a constant $L>0$ such that for any $x, y \in \bar G$ and any $\mu_1, \mu_2 \in {\cal P}(\bar G)$
\begin{align} 
       &|b(t, x,  \mu_1) - b(t, y,  \mu_2)| + |\sigma(t, x,  \mu_1) - \sigma(t, y,  \mu_2) | \label{eq:ass_new}\\
       &\leq L \bigg( |x-y|+ \sum_{j=1}^J
       \left |\int_G \phi_j(t,x,z)d \mu_1(z) - \int_G \phi_j(t,x,z)d \mu_2(z) \right | \bigg), \notag
   \end{align}
   where  $\phi_j(t,x,y)$, $j=1, \ldots, J$,  are  continuous  functions on $[0,T] \times \bar G \times \bar G$ and for some $C>0$
\[
\abs{\phi_j(t,x,y)-\phi_j(t,x,z)} \leq C \abs{y-z} .
\]
 Also, the coefficients $b(t, x, \mu)$ and $\sigma(t, x, \mu)$ are continuous in $t$. 
\end{assumption}

We will need the following elementary lemma. 
\begin{lemma}\label{prop_4.1}
    Let $\{A^{i}\}_{i=1}^{N}$ be a collection of independent $\mathbb{R}^{d}$-valued random variables with fourth bounded moment and $\mathbb{E}(A^{i}) = 0$, $i=1,\dots,N$. Then
    \begin{align}
        \mathbb{E}\bigg|\frac{1}{N}\sum_{i=1}^{N}A^{i}\bigg|^{4} \leq \frac{C}{N^{2}},
    \end{align}
    where $C>0$ is independent of $N$. 
\end{lemma}

Let us prove the following auxiliary lemma.
\begin{lemma}\label{lem:chaos1}
Let Assumptions~\ref{as:smooth} and \ref{as:new} be satisfied. Then the following bounds hold.

For all $i=1,\dots, N$,
   \begin{align}
       & \mathbb{E}|b(t,X^{i,N}(t),\hat{\mu}_{\textbf{X}^N (t)}) - b(t, X^{i}(t), \hat{\mu}_{\textbf{X} (t)})|^{4} 
       + \mathbb{E}|\sigma(t,X^{i,N} (t),\hat{\mu}_{\textbf{X}^N (t)})
       - \sigma(t, X^{i} (t), \hat{\mu}_{\textbf{X} (t)})|^{4} \nonumber \\ 
       &\leq C \bigg(\mathbb{E}|X^{i,N} (t)- X^{i} (t)|^{4} + \frac{1}{N}\sum_{j=1}^{N}\mathbb{E}|X^{j,N} (t) - X^{j} (t)|^{4}\bigg),
       \label{eq:lemchaos1}
   \end{align}
where $C >0$ is independent of $N$.
 
 For all $i=1,\dots, N$, 
   \begin{align}
       \mathbb{E}|b(t, X^{i} (t), \hat{\mu}_{\textbf{X} (t)}) - b(t,X^{i} (t),  \mathcal{L}^{X}_{t})|^{4} 
       + \mathbb{E}|\sigma(t, X^{i} (t), \hat{\mu}_{\textbf{X} (t)}) - \sigma(t,X^{i} (t),  \mathcal{L}^{X}_{t})|^{4} 
       \leq  \frac{C}{N^{2}}, 
       \label{eq:lemchaos2}
   \end{align}  
where $C >0$ is independent of $N$.
\end{lemma}
\begin{proof}
Thanks to Assumption~\ref{as:new}, we obtain
\begin{align*}
       &\mathbb{E}|b(t,X^{i,N}(t),\hat{\mu}_{\textbf{X}^N (t)}) - b(t, X^{i}(t), \hat{\mu}_{\textbf{X} (t)})|^{4} 
       + \mathbb{E}|\sigma(t,X^{i,N} (t),\hat{\mu}_{\textbf{X}^N (t)})
       - \sigma(t, X^{i} (t), \hat{\mu}_{\textbf{X} (t)})|^{4} \nonumber \\ 
       &\leq C \E|X^{i,N}(t)-X^{i}(t)|^4+ C \sum_{j=1}^J \E
       \left |\frac{1}{N}\sum_{l=1}^N \phi_j(t,X^{i}(t),X^{l,N}(t))  
       - \frac{1}{N}\sum_{l=1}^N \phi_j(t,X^{i}(t),X^{l}(t)) \right |^4 \\
       & \leq C \E|X^{i,N}(t)-X^{i}(t)|^4+ C \frac{1}{N}\sum_{l=1}^{N}\mathbb{E}|X^{l,N} (t) - X^{l} (t)|^{4},
   \end{align*}
hence we arrived at (\ref{eq:lemchaos1}).

Thanks to Assumption~\ref{as:new}, we get
\begin{align*}
& \mathbb{E}|b(t, X^{i} (t), \hat{\mu}_{\textbf{X} (t)}) - b(t,X^{i} (t),  \mathcal{L}^{X}_{t})|^{4} 
+ \mathbb{E}|\sigma(t, X^{i} (t), \hat{\mu}_{\textbf{X} (t)}) - \sigma(t,X^{i} (t),  \mathcal{L}^{X}_{t})|^{4} \\
 & \leq C \sum_{j=1}^J \E
       \left |\frac{1}{N} \sum_{l=1}^N \phi_j(t,X^{i} (t),X^{l} (t)) 
       - \int_G \phi_j(t,X^{i} (t),y)d \mathcal{L}^{X}_{t}(y) \right |^4 \\
  & \leq C \sum_{j=1}^J \mathbb{E}\bigg| \frac{1}{N}\sum_{l= 1}^{N}\phi_j(t,X^{i}(t), X^{l}(t))
    - \frac{1}{N-1}\sum_{l= 1, l\neq i}^{N}\phi_j(t,X^{i}(t), X^{l}(t))\bigg|^{4} \nonumber \\ 
    &  +  C \sum_{j=1}^J \mathbb{E}\bigg| \frac{1}{N-1}\sum_{l= 1, l\neq i}^{N}\phi_j(t,X^{i}(t), X^{l}(t))
    - \int_{G}\phi_j(t,X^{i}(t), y)\mathcal{L}^{X}_{t}(\d y)\bigg|^{4} \nonumber \\ 
    & \leq \frac{C}{N^{4}} +   
    C \sum_{j=1}^J\mathbb{E}\Bigg(\mathbb{E}\bigg| \frac{1}{N-1}\sum_{l= 1, l\neq i}^{N}\phi_j(t,x, X^{l}(t))
    - \int_{G}\phi_j(t,x, y)\mathcal{L}^{X}_{t}(\d y)\bigg|^{4}\Bigg| X^{i}(t) = x\Bigg).
\end{align*}
Note that for each $j$
\begin{align}
D_j^{l}(x) := \phi_j(t,x, X^{l}(t))- \int_{G}\phi_j(t,x, y)\mathcal{L}^{X}_{t}(\d y), \, l=1,\ldots, N, 
\end{align}
are independent random variables with bounded moments and $\E D^{l}_j(x)=0$.
Then, applying a conditional version of Lemma~\ref{prop_4.1}, we arrive at (\ref{eq:lemchaos2}).
\end{proof}

Now we proceed to the propagation of chaos theorem.

\begin{theorem}\label{thm:chaos}
Let Assumptions~\ref{as:smooth} and \ref{as:new} be satisfied. Then the following bound holds:
     \begin{align}\label{eq:chaosthm}
     \Big(\max_{i=1,\dots,N}\mathbb{E}\sup_{u \leq t}|X^{i,N}(u) - X^{i}(u)|^{4}\Big)^{1/4} \leq C \frac{1}{N^{1/2}},
     \end{align}
where $X^{i,N}$ represents the $i$-th particle among the interacting particles driven by (\ref{eq:particle_system}), $X^{i}$ is independent and identical copy of (\ref{eq:MVsde}), and $C >0$ is a constant independent of $N$. 
\end{theorem}
\begin{proof}
For brevity, we write $X^{i,N}_{t} := X^{i,N}(t)$ and $X^{i}_{t} := X^{i}(t)$. Analogously how (\ref{eq:process_bound0}) was derived, we can get
\begin{align}
            \E\sup_{u \leq t} \abs{X^{i,N}_{u} - X^{i}_{u}}^4 & \leq K \E \bigg ( \int_0^t \abs{X^{i,N}_{s} - X^{i}_{s}}^2 
            \abs{b(s,X^{i,N}_{s},\hat{\mu}_{\textbf{X}^N_s}) - b(s,X^{i}_{s},{\cal L}_s^X)}^2 \nonumber\\
             & \qquad + \abs{X^{i,N}_{s} - X^{i}_{s}}^2 
            \abs{\sigma (s,X^{i,N}_{s},\hat{\mu}_{\textbf{X}^N_s}) - \sigma (s,X^{i}_{s},{\cal L}_s^X}^2 \nonumber\\
            & \qquad + \abs{X^{i,N}_{s} - X^{i}_{s}}^4 + \abs{\sigma (s,X^{i,N}_{s},\hat{\mu}_{\textbf{X}^N_s}) 
            - \sigma (s,X^{i}_{s},{\cal L}_s^X) }^4 \d s \bigg). \label{eq:process_bound0NEW}
        \end{align}
We have
\begin{align}
    \abs{b(s,X^{i,N}_{s},\hat{\mu}_{\textbf{X}^N_s}) - b(s,X^{i}_{s},{\cal L}_s^X)}^4 
    &\leq 4( |b(s,X^{i,N}_{s},\hat{\mu}_{\textbf{X}^N_s}) - b(s, X^{i}_s, \hat{\mu}_{\textbf{X}_s})|^{4} \nonumber \\
    &+ | b(s, X^{i}_s, \hat{\mu}_{\textbf{X}_s}) - b(s,X^{i}_{s},{\cal L}_s^X) |^{4}), \label{eqn_5.9} \\ |\sigma(s,X^{i,N}_{s},\hat{\mu}_{\textbf{X}^N_s}) - \sigma(s,X^{i}_{s},{\cal L}_s^X)|^{4} 
    &\leq  4(|\sigma(s,X^{i,N}_{s},\hat{\mu}_{\textbf{X}^N_s}) - \sigma(s, X^{i}_s, \hat{\mu}_{\textbf{X}_s})|^{4} \nonumber \\
    &+ | \sigma(s, X^{i}_s, \hat{\mu}_{\textbf{X}_s}) - \sigma(s,X^{i}_{s},{\cal L}_s^X) |^{4}). \label{eqn_5.10} 
\end{align}
Using Young's inequality, (\ref{eqn_5.9}), (\ref{eqn_5.10}) and Lemma~\ref{lem:chaos1}, we obtain from (\ref{eq:process_bound0NEW}):
 \begin{align*}
     \E\sup_{u \leq t} \abs{X^{i,N}_{u} - X^{i}_{u}}^4 \leq C \int_{0}^{t}\mathbb{E}|  X^{i,N}(s) - X^{i}(s)|^{4} + \frac{1}{N}\sum\limits_{j=1}^{N}\mathbb{E}|X^{i,N}(s) - X^{i}(s)|^{4}ds + C\frac{1}{N^{2}}.
 \end{align*}
Hence
\begin{align*}
    \max_{i =1, \dots, N}\E\sup_{u \leq t} \abs{X^{i,N}_{u} - X^{i}_{u}}^4  
    \leq C\int_{0}^{t} \max_{i =1, \dots, N}\mathbb{E} \sup_{u \leq s}|X^{i,N}(u) - X^{i}(u)|^{4} ds + C\frac{1}{N^{2}},
\end{align*}
which on applying Gronwall's lemma gives the desired result. 
\end{proof}

It is not difficult to verify that the CBO models from Sections~\ref{subsec_2.1} and \ref{subsec_2.2} satisfy Assumption~\ref{as:new}.
Consequently, we get the following corollary.
\begin{corollary}
The particle systems (\ref{particle_cbo}) and (\ref{CBO_func}) converge to their mean-field limits (\ref{eq:mf_cbo}) and (\ref{eq:mf_cbo_repul}), respectively: for them the inequality (\ref{eq:chaosthm}) holds.    
\end{corollary}

\section{Long-time behavior of reflected mean-field models}\label{sec:longtime}

In Section~\ref{subsec_reflection_coupling_conv}, we use the coupling technique to study long time behavior of the reflected McKean-Vlasov SDEs (\ref{eq:MVsde}) with $\sigma \equiv I$. In Section~\ref{sec:globminim}, we show convergence of CBO models from Section~\ref{sec:rcbo} towards the global minimum.

\subsection{Convergence of reflected mean-field  SDEs with additive noise} \label{subsec_reflection_coupling_conv}

Consider the reflected McKean-Vlasov SDEs (\ref{eq:MVsde}) with $\sigma \equiv I$:
\begin{align}
                dX(t) = b(X(t), \mathcal{L}_{X(t)}) \d t +  \d W(t) +  \nu(X(t)) \d L(t), \quad X(0) \in \bar{G}.  
                \label{eq:mvsde5}
\end{align}
Here, we provide a quantitative non-asymptotic bound for the exponential convergence of $\mathcal{L}_{X(t)}$, i.e. law of $X(t)$  to the invariant measure $\tilde{\mu}$ defined on $\bar{G}$.   

Recall $c := 1/(2R_0)$ where $R_0$ is the radius of uniform exterior sphere (see (\ref{eq:ext_sphere_condition})). We denote 
\begin{align}
     R_1 =  \sup_{x, y \in \bar{G}}|x-y| \,\, \text{ and }  \,\, \hat{\lambda}_c = 1 - 2 c R_1.
\end{align}
 Note that for a convex domain $c = 0$.

\begin{assumption} \label{ass:sec51}
    The constant $\hat{\lambda}_c>0$.
\end{assumption}

This assumption allows for the domain $G$ to be non-convex but it is weaker than the uniform exterior sphere condition (\ref{eq:ext_sphere_condition}). 

As in the proof of Lemma~\ref{lemma:contraction}, $d(\cdot, \partial G)$ is the distance to the boundary of $G$, defined on some neighborhood of $\partial G$ within $G$ and the function $g(\cdot)$ is a smooth, bounded extension of $d$ to the whole space so that 
\begin{align}
    \max_{x \in \bar G} |\nabla g (x)| =1.
\end{align}
 We also use the notation
 \begin{equation}
     g_{\min} = \min_{x \in \bar{G}} g(x) \quad \text{and}  \quad  g_{\max} = \max_{x \in \bar{G}} g(x).
 \end{equation}

We employ reflection coupling (see \cite{lindvall1986coupling, chen1989coupling, eberle2016reflection}) to establish $1$-Wasserstein bound between the law of $X(t)$ and its invariant measure. 
Using reflection coupling, the convex domain case was considered  in \cite[Remark~3]{eberle2016reflection} for additive noise and in \cite{wang2023exponential} for multiplicative noise. 
 
\begin{theorem}
    Let Assumption~\ref{as:smooth}, \ref{as:new} and~\ref{ass:sec51} hold.  Then the following bound for the solution of (\ref{eq:mvsde5}) holds:
    \begin{align}
      \mathcal{W}_1(\mathcal{L}_{X(t)}, \tilde{\mu}) \leq C_{\textrm{mf }} \mathcal{W}_1(\mathcal{L}_{X(0)}, \tilde{\mu}) e^{-C_0 t},
    \end{align}
    where 
    \begin{align}
        C_{\text{mf}} & =   2\exp\bigg(  \frac{K_1}{K_3} \frac{R_1^2}{2} + \frac{K_2}{K_3} R_1\bigg)  \exp(2 c g_{\max}) , \\
       C_0 &=  \frac{K_3}{2 \int_0^{R_1} \int_0^r\exp\bigg(  \frac{K_1}{K_3} \frac{r^2-s^2}{2} + \frac{K_2}{K_3} (r-s) \bigg) ds dr} , \notag
    \end{align}
with
\begin{align}
    K_1 &=  L + c \sup_{x \in \bar{G}, \mu \in \mathcal{P}(\bar{G})}( 2 |b(x, \mu)|  + |\Delta g(x)|  ) + 2c^2, \\ 
    K_2  &=  4 c e^{-2 c g_{\min}} + LR_1 e^{-2 c g_{\min}}, \quad 
    K_3 = 2 \hat{\lambda}_c e^{-4c g_{\max}}, \notag
\end{align}
and $L$ is from (\ref{eq:ass_new}).
\end{theorem}
 
\begin{proof}
Consider the coupling 
\begin{align}
    dX(t) &= b(X(t), \mathcal{L}_{X(t)})dt +  dW(t) + \nu(X(t))dL^{X}(t), \;\; t>0,  \\
    dY(t) &= b(Y(t), \mathcal{L}_{Y(t)} )dt + (I - 2E(t)E^{\top}(t))dW(t) + \nu(Y(t))dL^{Y}(t), \; \tau >t,
\end{align}
and 
\begin{align*}
    X(t) = Y(t), \quad \tau \leq t,
\end{align*}
where the coupling time $\tau$ is
\begin{align}
    \tau := \inf\{ t>0\; ; \; X(t) = Y(t)\},
\end{align}
and $E(t) = (X(t) - Y(t))/|X(t) - Y(t)|$.  
Note that for every $s \ge 0$ the matrix $I- 2E(t)E^{\top}(t) I(\tau>t)$ is orthogonal. 

In what follows, for brevity, we will denote $X_{t} := X(t)$, $Y_{t} := Y(t) $, $W_{t} = W(t)$, $L^X_t := L^X(t)$, $L^Y_t := L^Y(t)$, $E_t:=E(t)$,  $b^X_t := b(X_t, \mathcal{L}_{X_t})$, and $b^Y_t := b(Y_t, \mathcal{L}_{Y_t})$. Denote $Z_{t} := g(X_{t}) + g(Y_t)$. Introduce 
for $\tau >t$: $r_t = \exp(-c Z_{t})|X_{t} - Y_{t}|$ and for $\tau \leq t$: $r_t=0$.

By Ito's formula, we have for $\tau >t$:
\begin{align*}
    d r_t & = \exp(-c Z_{t}) \Big[( E_{t} \cdot (b^{X}_{t} - b^{Y}_{t})) dt + 2 E_t^{\top} dW_{t}   
     + |X_{t} - Y_{t}| \big[ cF_{t} dt  + \frac{c^{2}}{2}H_t - c (\nabla g(X_{t}) \cdot dW_{t}) \\   
    &   - c  (\nabla g(Y_{t}) \cdot( I - 2E_{t}E_{t}^{\top})dW_{t})  \big] - 2c\nabla ^{\top} g(X_{t}) E_t dt 
     + 2c  \nabla ^{\top}  g(Y_{t}) E_t dt \\
    & + \big\{\big((E_{t} \cdot\nu(X_{t})) -c |X_{t} - Y_{t} |\big) dL^{X}_{t}  -\big((E_{t} \cdot \nu(Y_t)) +c |X_{t} - Y_{t} |\big) dL^{Y}_{t}\big\}\Big],
\end{align*}
where $F_{t} :=  -  [(\nabla g(X_{t}) \cdot b^{X}_{t}) + (\nabla g(Y_{t}) \cdot b^{Y}_{t}) + \frac{1}{2}\Delta g(X_{t}) + \frac{1}{2}\Delta g(Y_{t})]  $ and $H_t := |\nabla g(X_t) + (I - 2 E_t E_t^{\top})\nabla g(Y_t)|^2 $. 

Consider an increasing concave function $\psi : [0,R_1] \rightarrow \R_+$ with $\psi(0)=0$, whose precise choice will be made later. Using Ito's formula, we get for $\tau >t$:
\begin{align*}
     d\psi(r_t) =  \psi'(r_t) dr_t + \frac{1}{2}  \psi''(r_t) d[r, r]_{t},
\end{align*}
where $[r, r]_{t} $ denotes a bracket process which is a non-decreasing process satisfying for $ \tau >t$:
\begin{align*}
    d [r_t,  r_t] & = e^{-2 c Z_t } \big| - c |X_t - Y_t |\nabla g(X_t)  -  c |X_t - Y_t| (I - 2 E_t E_t^\top)\nabla g(Y_t)   + 2E_t\big|^2 dt \numberthis     \\  
    &  = 4e^{-2 c Z_t } dt +  e^{-2 c Z_t }c^{2} |X_t - Y_t|^2  | \nabla g(X_t) + (I - 2 E_t E_t^\top)\nabla g(Y_t)|^2 \\  
    &   - 4 |X_t - Y_t|c e^{-2 c Z_t }\big( (\nabla g(X_t) \cdot E_t) - (\nabla g(Y_t) \cdot E_t)\big)  \\  
    &  \geq 4e^{-2 c Z_t } dt -  4 e^{-2 c Z_t }c\big( (\nabla g(X_t) - \nabla g(Y_t)) \cdot (X_t- Y_t)\big)dt \\   
    &  \geq 4 e^{-2 c Z_t}\big( 1 - 2 c R_1  \big) dt = 4 e^{-2 c Z_t} \hat{\lambda}_c dt . 
\end{align*}
Let $F_{\max} := \sup_{x \in \bar{G}, \mu \in \mathcal{P}(\bar{G})}( 2 |b(x, \mu)|  + |\Delta g(x)|  ).$
Using (\ref{eq:ext_sphere_condition}) and that $\max_{x \in \bar{G}} |\nabla g (x)| = 1$ , $\psi ' >0$ and $\psi'' <0$, we get for $\tau >t$:
\begin{align*}
     d\psi(r_t) &\leq \psi'(r_t) \exp(-c Z_t)|b^X_t - b^Y_t|dt + c \psi'(r_t) r_t \big( F_{\max} + 2c\big)dt 
     +  4c\exp(-c Z_t)\psi'(r_t)\tilde{g}_{\max} dt
       \\  &  \quad 
     + 2\hat{\lambda}_c \psi''(r_t)\exp(-2c Z_{t})dt 
   + \psi'(r_t) \exp(-c Z_t) \big( 2(E_t \cdot dW_t) - c (\nabla g(X_t) \cdot d W_t) 
   \\  & \quad 
   - c (\nabla g(Y_t) \cdot (I - 2 E_t E_t^{\top}) dW_t) \big) .       
\end{align*}
Let us denote the martingale term for $\tau >t$ as
\begin{align*}
M_t &=   \int_0^t \psi'(r_u) \exp(-c Z_u) (\big( 2E_u - c \nabla g(X_u)   
   - c (I - 2 E_u E_u^{\top})\nabla g(Y_u)\big) \cdot  dW_u) . 
\end{align*}
Exploiting Assumption~\ref{as:new}, we obtain
 \begin{align*}
     \E\psi(r_{t \wedge \tau}) &\leq  \E \psi(r_0) + \E\bigg\{ \int_0^{t \wedge \tau}\Big(\psi'(r_u) (Lr_u + \exp(-c Z_u)L\E|X_u - Y_u|)+  c \psi'(r_u) r_u \big( F_{\max} + 2c\big) \\
      &  +  4c\exp(-c Z_u)\psi'(r_u) 
     + 2\hat{\lambda}_c\psi''(r_u)\exp(-2c Z_{u})\Big) du\bigg\} + \E M_{t \wedge \tau} \\
      &   \leq \E \psi(r_0)+ \E\bigg\{\int_0^{t \wedge \tau}\psi'(r_u) ( K_1 r_u  + K_2)  +  \psi''(r_u) K_3 du \bigg\} , \numberthis \label{rcbo_equation_5.7}
\end{align*}
where we have used Doob's optional stopping theorem to get $\E M_{t \wedge \tau} = 0$.

Now we aim to find  an increasing concave function $\psi$ such that for $r \in [0,R_1]$:
\begin{align}
    \psi'(r)(K_1 r + K_2)   +  K_3\psi''(r) \leq - C_0\psi(r), \label{rcbo_equation_5.8}
\end{align}
for some constant $C_0 >0$ independent of $r$.  
To this end, we make the choice of derivative of $\psi$ inspired from \cite{eberle2016reflection} as:
\begin{align*}
    \psi'(r) = \psi_1(r) \psi_2(r),
\end{align*}
where 
\begin{align*}
    \psi_1(r) = \exp\bigg(- \frac{K_1}{K_3} \frac{r^2}{2} - \frac{K_2}{K_3} r\bigg) 
\end{align*}
and
\begin{align}
   \psi_2(r) =  1 -  \frac{ C_0}{K_3} \int_{0}^{r}\frac{1}{\psi_1(s)} \int_{0}^{s}\psi_1(s') ds'ds  
\end{align}
with 
\begin{align*} 
  C_0 =   \frac{K_3}{2 \int_{0}^{R_1}\frac{1}{\psi_1(s)} \int_{0}^{s}\psi_1(s') ds'ds } .
\end{align*}
It is not difficult to verify that $\psi(r)=\int_0^r \psi_1(s) \psi_2(s) ds$ satisfies (\ref{rcbo_equation_5.8}).
Further, $1/2 \leq \psi_2(r) \leq 1$ and hence
\begin{align}
   \frac{r}{2} \psi_1(R_1) \leq \psi(r) & \leq r. \label{eq:psile}
\end{align}

From (\ref{rcbo_equation_5.7}) and (\ref{rcbo_equation_5.8}), taking into account that $r(\tau)=0$ and $\psi(0)=0$, we get
\begin{align*}
  \E\big( \psi(r_{t})  I(  \tau \geq t)\big) =\E \psi(r_{t \wedge \tau}) \leq  \E\psi(r_0) - C_0\E \int_{0}^{t \wedge \tau} \psi(r_u)du =  \E\psi(r_0) - C_0\int_{0}^{t }\E \big( \psi(r_u) I(  \tau \geq t) \big) du.
\end{align*}
Therefore, $\E\big( \psi(r_{t})I(  \tau \geq t)\big)  \leq   \E\psi(r_0) e^{-C_0 t}.$
 Noting that $ \E\big( \psi(r_{t \wedge \tau})  I(  \tau \leq t)\big) = 0$, we have
\begin{align*}
 \E \psi(r_{t}) \leq   \E\psi(r_0) e^{-C_0 t}.
\end{align*} 
Then, thanks to (\ref{eq:psile}), we get 
\begin{align}
    \E r_t \leq 2 \psi_1^{-1}(R_1)\E\psi(r_0) e^{-C_0 t}
\end{align}
and, since $ \exp(-2 c g_{\max} ) |X_t- Y_t| \leq r_t $, we arrive at 
\begin{align*}
    \E |X_t- Y_t| \leq 2 \exp(2cg_{\max}) \psi_1^{-1}(R_1)\E|X_0 - Y_0|e^{- C_0 t} . 
\end{align*}
\end{proof}


\subsection{Convergence of mean-field limit of CBO-type models to the global minimum}\label{sec:globminim}

This section is devoted to showing convergence of constrained CBO models from Sections~\ref{subsec_2.1} and \ref{subsec_2.2} towards global minimizers. Similar to \cite{carrillo_analytical_2018, carrillo_consensus-based_2021} (see also \cite{tretyakov_consensus-based_2023}), the main tool which we will utilize is the Laplace principle which states that for any compactly supported measure $\mu$ and positive function $f$ the following holds:  
    \begin{align}
        \lim_{\alpha \rightarrow \infty}- \frac{1}{\alpha}\log \int_{\R^d} e^{-\alpha f(x)} \mu(dx)  = f(x_{\min})
    \end{align}
    with $x_{\min}$ being the minimizer of $f(x)$ defined on the support of  $\mu$.
We will use the notation $$ \osc = f_{\max} - f_{\min}.$$ 

The analysis in this section is done for convex domains. 

\begin{assumption} \label{convex_assum}
     $G$  is a convex bounded domain.
\end{assumption}

Let us explain why we restrict analysis in this section to convex domains. In the case of non-convex domains, CBO models with constraints of the type presented in Sections~\ref{subsec_2.1} and~\ref{subsec_2.2} face two difficulties. The first difficulty concerns the proof techniques used for showing convergence to the global minimum which has been used so far in the CBO literature (see, e.g. \cite{carrillo_analytical_2018, carrillo_consensus-based_2021,tretyakov_consensus-based_2023,Massimo24} and references therein and also the proofs below in this section). Roughly speaking, noise does not play a positive role within those proofs which rely on the drift driving solutions of the mean-field SDEs to the consensus. It works well in the convex case and in the whole Euclidean space. However, in situations like the one illustrated in  
Figure~\ref{fig:ccbo_non_convex_example}, the drift forces trajectories to go outside the domain rather than driving it to a consensus and, hence, we need to rely on the noise term (via effectively large deviations) to get trajectories to a proximity of the global minimum, which happens with a positive probability, and then the drift can drive trajectories to a consensus. Hence, a probabilistic proof needs to be developed for proving convergence to the global minimum, which is an interesting problem for future research. Such a proof would also be able to provide more insight on how convergence to the global minimum depends on the noise strength which would be useful even in the whole Euclidean space.
The second difficulty is more fundamental as it is related to the design of CBO models themselves. In the example of Figure~\ref{fig:ccbo_non_convex_example}, the drift forces trajectories the wrong way because its form is based on the Euclidean metric. Therefore, it is of interest to find a different formulation of CBO models which is to imbibe more information of the domain (e.g., by using a non-Euclidean metric) but at the same time which remains computationally efficient. Constructing such models is out of scope of this work. Nevertheless, we tested the CBO model of Section~\ref{subsec_2.1} on a non-convex function with non-convex domain (see Section~\ref{subsec:towsend}) and the CBO model demonstrated excellent performance. 

\begin{figure}[ht]
    \centering
    \includegraphics[width=0.35\linewidth]{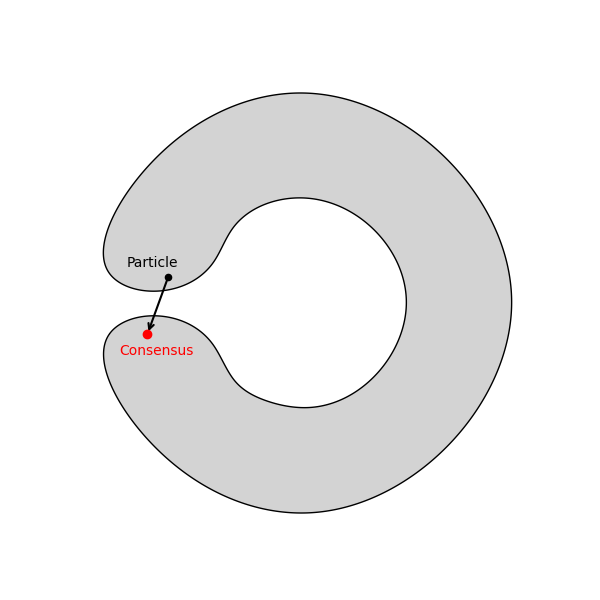}
    \caption{The example of a non-convex domain, where the Euclidean path of a particle to the consensus does not lie within the domain. In the convex case, this behavior cannot occur.}
    \label{fig:ccbo_non_convex_example}
\end{figure}

In Theorem~\ref{subsec_ref_cbo}, we prove convergence of the reflected CBO model from Section~\ref{subsec_2.1} towards the global minimum and consider the case of CBO with repelling force from Section~\ref{subsec_2.2} in Theorem~\ref{subsec_ref_cbo_rep}.
We note that under Assumptions~\ref{as:lipschitz} and \ref{convex_assum} well-posedness of (\ref{eq:MVsde}) and the corresponding propagation of chaos were proved in \cite{adams_large_2022}. 

\begin{theorem}\label{subsec_ref_cbo}
    Let Assumptions~\ref {as:objective} and \ref{convex_assum} hold. If the following condition is satisfied 
    \begin{align*}
        \eta_0 := 2\beta - \sigma^{2} \big( 1 +  e^{ 2 \alpha \osc} \big) >0 
    \end{align*}
    then there exists an $x^* \in \bar{G}$ such that $X(t) \rightarrow x^{*}$ as $t\rightarrow \infty$ and
    \begin{align*}
    f(x^{*}) \leq f_{\min} + \Gamma(\alpha),
\end{align*}
where $X(t)$ is the solution of the mean-field SDE (\ref{eq:mf_cbo}) and  $\Gamma(\alpha) \rightarrow 0$ as $\alpha \rightarrow \infty$. 
\end{theorem}
\begin{proof}
Using Ito's formula, we get
\begin{align*}
    |X(t) &- \E X(t) |^{2} = |X(0) - \E X(0)|^{2} - 2 \int_{0}^{t} \beta (X(s) - \E X(s)) \cdot (X(s) - \bar{X}(s)) ds 
    \nonumber 
    \\  & \;\;\;
 - 2\int_{0}^{t} (X(s) - \E X(s)) d\E X(s)
 + 2 \int_{0}^{t}\sigma \big((X(s) - \E X(s)) \cdot \diagon(X(s) - \bar{X}(s)) dW(s)\big) 
    \nonumber \\ & \;\;\;
    + \int_{0}^{t} \sigma^{2} |X(s) - \bar{X}(s)|^{2} ds + 2 \int_{0}^{t} ((X(s) - \E X(s)) \cdot \nu(X(s))) dL(s) 
    \nonumber \\ & 
= |X(0) - \E X(0)|^{2} -    2\int_{0}^{t}\beta|X(s) - \E X(s)|^{2} ds  
\\ & \;\;\;
 - 2\int_{0}^{t} (X(s) - \E X(s)) d\E X(s) 
-2\int_{0}^{t}\beta(( X(s) - \mathbb{E} X(s)))\cdot (\E X(s) - \bar{X}(s)))ds  
\\ & \;\;\;
+   \int_{0}^{t}\sigma^{2}|X(s) - \bar{X}(s)|^{2} ds 
 +  2 \int_{0}^{t}\sigma \big((X(s) - \E X(s)) \cdot \diagon(X(s) - \bar{X}(s)) dW(s)\big) 
 \\  &  \;\;\; 
 + 2 \int_{0}^{t} ((X(s) - \E X(s)) \cdot \nu(X(s))) dL(s) .
\end{align*}
  Due to Assumption~\ref{convex_assum}, we have
\begin{align*}
    ((X(s) - \E X(s)) \cdot \nu(X(s))) \leq 0.
\end{align*}
Also, note that $\E(( X(s) - \mathbb{E} X(s))\cdot (\E X(s) - \bar{X}(s))) =  0$.
Consequently, we get
\begin{align*}
     \var(t):=& \E|X(t) - \E X(t) |^{2} \leq  \E|X_{0} - \E(X_{0})|^{2} 
      -    2\int_{0}^{t}\beta\E|X(s) - \E X(s)|^{2} ds   \\
&+   \int_{0}^{t}\sigma^{2}\E|X(s) - \bar{X}(s)|^{2} ds.
\end{align*}

We have
\begin{align}
    \E|X(s) - \bar{X}(s)|^{2} = \E|X(s) - \E X(s)|^{2} +  |\E X(s) - \bar{X}(s)|^{2} \label{rccbo_equan_5.2}
\end{align}
and
\begin{align*}
     &|\mathbb{E}X(s^{}) - \bar{X}(s^{})|^{2} = \bigg| \mathbb{E}X(s^{}) - \frac{\mathbb{E}X(s^{})e^{-\alpha f(X(s^{}))}}{\mathbb{E}e^{-\alpha f(X(s^{}))}}\bigg|^{2} =
    \bigg|\mathbb{E} \bigg( \Big(\mathbb{E}X(s^{}) - X(s^{})\Big)\nonumber 
    \times \frac{e^{-\alpha f(X(s^{}))}}{\mathbb{E}e^{-\alpha f(X(s^{}))} }\bigg)\bigg|^{2}\\ 
    & \;  \leq  e^{ 2 \alpha \osc} \E|X(s) - \E X(s)|^{2}. \numberthis  \label{rccbo_equan_5.3}
\end{align*}
Therefore,
\begin{align*}
    \frac{d}{dt} \var(t) \leq    \Big(- 2\beta + \sigma^{2} \big( 1 +  e^{ 2 \alpha \osc} \big)\Big) \var(t)
        \leq \Big(- 2\beta + \sigma^{2} \big( 1 +  e^{ 2 \alpha \osc} \big)\Big) \var(t), \numberthis \label{equation_rcbo_5.4}
\end{align*}
which implies
\begin{align*}
    \var(t) \leq \var(0) e^{-\eta_0 t}. \numberthis
\end{align*}
This also means, due to (\ref{rccbo_equan_5.2}) and (\ref{rccbo_equan_5.3}), there exist positive constants $c_1 $ and $c_2$ independent of $t$ such that
\begin{align}
     \E|X(t) - \bar{X}(t)|^{2} &\leq c_1 e^{-\eta_0 t}, \label{eqn_rcbo_5.8} \\ 
     |\E X(t) - \bar{X}(t)|^{2} & \leq c_2 e^{-\eta_0 t}. \label{eqn_rcbo_5.9}
\end{align}

Let us fix a constant $q \in G$. Then, again using Ito's formula, we have
\begin{align*}
    d|X(t) -q |^{2} &= -2 \beta \big( ( X(t) - q )\cdot (X(t) - \bar{X}(t))\big)dt + \sigma^{2}|X(t) - \bar{X}(t)|^{2} dt   
    \\   &  \quad 
    + 2 \sigma \big( ( X(t) - q )\cdot \diagon(X(t) - \bar{X}(t))dW(t)\big) + 2 ((X(t) - q )\cdot \nu(X(t))) dL(t). 
\end{align*}
Using Tanaka's trick (thanks to Assumption~\ref{convex_assum}) and taking expectation on both sides, we have
\begin{align*}
    d\E|X(t) -q |^{2} &\leq -2 \beta \E\big( ( X(t) - q )\cdot (X(t) - \bar{X}(t))\big)dt + \sigma^{2} \E|X(t) - \bar{X}(t)|^{2} dt  
    \\  & 
    \leq -2 \beta \E\big( ( X(t) - q )\cdot (X(t) - \bar{X}(t))\big)dt + \sigma^{2} c_1 e^{-\eta_0 t} dt, 
\end{align*}
where we have used (\ref{eqn_rcbo_5.8}) in the last inequality. Splitting the term $\big( ( X(t) - q )\cdot (X(t) - \bar{X}(t))\big)$, we get
\begin{align*}
    \frac{d}{dt}\E|X(t) -q |^{2} &\leq 
  -2 \beta \E\big( ( X(t) - q )\cdot (X(t) - \bar{X}(t))\big) + \sigma^{2} c_1 e^{-\eta_0 t} 
    \\ & \leq -2 \beta \E| X(t) - q |^{2} + 2 \beta ((\E X(t) - q)\cdot ( \bar{X}(t) -q)) + \sigma^{2} c_1 e^{-\eta_0 t}. 
\end{align*}
Rewriting $ ((\E X(t) - q)\cdot ( \bar{X}(t) -q)) = |\E X(t) - q|^{2}  +  ((\E X(t) - q)\cdot ( \bar{X}(t) - \E X(t)))$ gives
\begin{align*}
    \frac{d}{dt}\E|X(t) -q |^{2} &\leq -2 \beta \E| X(t) - q|^{2} + 2\beta|\E X(t) - q|^{2} + 2\beta((\E X(t) - q)\cdot ( \bar{X}(t) - \E X(t))) +  \sigma^{2} c_1 e^{-\eta_0 t}.
\end{align*}
 Due to Jensen's inequality, we get $|\E X(t) - q|^{2} \leq \E |X(t) - q|^{2}$
which provides the following differential inequality:
\begin{align*}
    \frac{d}{dt}\E|X(t) -q|^{2} &\leq  2\beta((\E X(t) - q)\cdot ( \bar{X}(t) - \E X(t))) +  \sigma^{2} c_1 e^{-\eta_0 t}
    \\  & 
    \leq  2\beta \sqrt{c_{2}} |\E X(t) - q| e^{-\eta_0 t/2}  +  \sigma^{2} c_1 e^{-\eta_0 t}.
\end{align*}
Note that $|\E X(t) - q|  $ is bounded uniformly in $t$ due to boundedness of domain $G$. 
 Therefore, we can say that there exists a constant $c_3>0$ such that
\begin{align}
    \frac{d}{dt}\E|X(t) -q |^{2} \leq c_3 e^{-\eta_0 t/2}, \label{equation_rcbo_5.7}
\end{align}
which together with $|\E X(t) -q |^{2} \leq \E|X(t) -q |^{2}$ implies that there exists an $x^* \in \bar{G}$ such that
$\E X(t)$ is converging to $x^*$ as $t \rightarrow \infty$. Here, $x^*$ belongs to $\bar{G}$ due to convexity of $G$. 

Using Markov's inequality, we get
\begin{align*}
    \mathbb{P}(|X(t) - \mathbb{E}X(t)| \geq e^{-\eta_0 t/4}) \leq \frac{\var{(t)}}{e^{- \eta_0 t/2}} \leq C e^{- \eta_{0} t/2},
\end{align*}
where $C$ is a positive constant independent of $t$. Then, using the Borel-Cantelli lemma, $|X(t) - \mathbb{E}X(t)| \rightarrow 0$ as $t \rightarrow \infty$ a.s., and hence  $X(t) \rightarrow x^{*}$ a.s. Consequently, applying the bounded convergence theorem, we arrive at $\mathbb{E}e^{-\alpha f(X(t))} \rightarrow e^{-\alpha f(x^{*})} $ as $t \rightarrow \infty$.   Therefore, we can say: for all $\epsilon >0$ there exists a $T^* >0$ such that
 \begin{align}
     \big|\mathbb{E}e^{-\alpha f(X(t))} - e^{-\alpha f(x^{*})}\big| \leq \epsilon  \label{rcbo_eqn_5.5}
 \end{align}
for all $ t \geq T^*$. Since the probability distribution of $X(T^*)$ is compactly supported, using the Laplace principle we obtain
\begin{align}
    - \frac{1}{\alpha}\log(\E e^{-\alpha f(X(T^*))} )   
    \leq 
    f_{\min} + \Gamma_{1}(\alpha),   \label{rcbo_eqn_5.6}
\end{align}
where $\Gamma_{1}(\alpha) \rightarrow 0$ as $\alpha \rightarrow \infty$. And, from (\ref{rcbo_eqn_5.5}), we have
\begin{align*}
    \E e^{-\alpha f (X(T^*))} \leq \epsilon + e^{-\alpha  f(x^*)},
\end{align*}
which implies
\begin{align}
     \log \E e^{-\alpha f (X(T^*))} \leq \log\big(\epsilon + e^{-\alpha  f(x^*)}\big) \,\, \text{ and }  \,\,
     -\frac{1}{\alpha} \log \E e^{-\alpha f (X(T^*))} \geq -\frac{1}{\alpha}\log\big(\epsilon + e^{-\alpha  f(x^*)}\big).  \label{rcbo_eqn_5.7}
\end{align}
Using (\ref{rcbo_eqn_5.6}) and (\ref{rcbo_eqn_5.7}), we obtain
\begin{align}
   -\frac{1}{\alpha}\log\big(\epsilon + e^{-\alpha  f(x^*)}\big) \leq  f_{\min} + \Gamma_{1}(\alpha).
\end{align}
From the mean value theorem, we have
\begin{align}
    \log\big(\epsilon + e^{-\alpha  f(x^*)}\big) = - \alpha f(x^{*}) +  \frac{\epsilon}{\gamma \epsilon + e^{-\alpha f(x^{*})}},
\end{align}
for some $\gamma \in (0,1)$. Therefore, 
\begin{align*}
    f(x^{*}) \leq f_{\min} + \Gamma_{1}(\alpha) +  \frac{\epsilon}{\alpha} \frac{1}{\gamma \epsilon + e^{-\alpha f(x^{*})}}.
\end{align*}
\end{proof}

\begin{theorem}\label{subsec_ref_cbo_rep}
     Let Assumptions~\ref {as:objective} and \ref{convex_assum} hold. Let
    \begin{align*}
        \eta_1 := - 2\beta + 3 \lambda(0) + \sigma^{2} \big( 1 +  e^{ 2 \alpha \osc}\big)
    \end{align*}
    be strictly greater than a positive constant $K_1$. If $ \lambda(t) \leq K_2/(1 + t^{2}) $ for some $ K_2 >0$, 
    then there exists an $x^* \in \bar{G}$ such that $X(t) \rightarrow x^{*}$ as $t\rightarrow \infty$ and
    \begin{align*}
    f(x^{*}) \leq f_{\min} + \Gamma(\alpha),
\end{align*}
where $\Gamma(\alpha) \rightarrow 0$ as $\alpha \rightarrow \infty$. 
\end{theorem}

\begin{proof}
Using Ito's formula, we have
\begin{align*}
    |X(t) &- \E X(t) |^{2} = |X(0) - \E X(0)|^{2} - 2 \int_{0}^{t} \beta (X(s) - \E X(s)) \cdot (X(s) - \bar{X}(s)) ds 
    \nonumber 
    \\  & 
 \;\;\; + 2\int_{0}^{t}\lambda(s) \int_{\mathbb{R}^{d}}\big((X(s) - \mathbb{E}X(s)) \cdot (X(s)- y)\big)\exp{\Big(-\frac{1}{2}| X(s)  - y|^{2}\Big)}\mathcal{L}_{X(s)}(\d y) ds \nonumber 
    \\  & \;\;\;
    + 2 \int_{0}^{t}\sigma \big((X(s) - \E X(s)) \cdot \diagon(X(s) - \bar{X}(s)) dW(s)\big) 
    \nonumber \\ & \;\;\;
    + \int_{0}^{t} \sigma^{2} |X(s) - \bar{X}(s)|^{2} ds + 2 \int_{0}^{t} ((X(s) - \E X(s)) \cdot \nu(X(s))) dL(s).
\end{align*}
The objective here is to deal with the repelling term. In that pursuit, we apply Young's inequality to get
\begin{align*}
    \E\int_{\mathbb{R}^{d}}\big((X(s) & - \mathbb{E}X(s)) \cdot (X(s)- y)\big)\exp{\Big(-\frac{1}{2}| X(s)  - y|^{2}\Big)}\mathcal{L}_{X(s)}(\d y) 
    \\  &\leq  \frac{1}{2}\E|X(s)  - \mathbb{E}X(s)|^{2} +  \frac{1}{2} \E\int_{\mathbb{R}^{d}} |X(s)- y|^{2}\exp{\Big(-\frac{1}{2}| X(s)  - y|^{2}\Big)}\mathcal{L}_{X(s)}(\d y) 
     \\  &\leq   \frac{1}{2}\E|X(s)  - \mathbb{E}X(s)|^{2} +  \frac{1}{2} \E\int_{\mathbb{R}^{d}} |X(s)- y|^{2}\mathcal{L}_{X(s)}(\d y)  
     \\ & \leq \frac{3}{2}\E|X(s)  - \mathbb{E}X(s)|^{2}, \numberthis \label{eqn_rcbo_5.12}
\end{align*}
since $  \E\int_{\mathbb{R}^{d}} |X(s)- y|^{2}\mathcal{L}_{X(s)}(\d y)  = 2 \E |X(s) - \E X(s)|^{2}$. 

Exploiting similar arguments as the ones used to get (\ref{equation_rcbo_5.4}) and using the inequality obtained in (\ref{eqn_rcbo_5.12}), we ascertain
\begin{align*}
    \frac{d}{dt}  \var(t) &\leq    \Big(- 2\beta + 3\lambda(0) + \sigma^{2} \big( 1 +  e^{ 2 \alpha \osc} \big)\Big)  \var(t) .
\end{align*}
With an appropriate choice of $\beta$, $\sigma$ and of decreasing $ \lambda(t)$, we have
\begin{align*}
    \frac{d}{dt}  \var(t) \leq c_4 e^{-\eta_1 t},
\end{align*}
where $c_4$ and $\eta_1$ are independent of $t$.

Let us again fix a constant $q \in G$. Applying Ito's formula gives 
\begin{align*}
    d|X(t) -q |^{2} &= -2 \beta \big( ( X(t) - q )\cdot (X(t) - \bar{X}(t))\big)dt + \sigma^{2}|X(t) - \bar{X}(t)|^{2} dt  
    \\  & \quad 
     + 2\lambda(t)\int_{\mathbb{R}^{d}}\big((X(t)  -q) \cdot (X(t)- y)\big)\exp{\Big(-\frac{1}{2}| X(t)  - y|^{2}\Big)}\mathcal{L}_{X(t)}(\d y)
    \\   &  \quad 
    + 2 \sigma  \big( ( X(t) - q )\cdot \diagon(X(t) - \bar{X}(t))dW(t)\big) + 2 (X(t) - q \cdot \nu(X(t))) dL(t). 
\end{align*}
Using Young's inequality, we have
\begin{align*}
    \int_{\mathbb{R}^{d}}\big((X(t) & -q) \cdot (X(t)- y)\big)\exp{\Big(-\frac{1}{2}| X(t)  - y|^{2}\Big)}\mathcal{L}_{X(t)}(\d y) 
    \\   & = \frac{1}{2}|X(t)  -q|^{2} \int_{\mathbb{R}^{d}}\exp{\Big(-\frac{1}{2}| X(t)  - y|^{2}\Big)}\mathcal{L}_{X(t)}(\d y) 
    \\  & \quad 
    +    \frac{1}{2}\int_{\mathbb{R}^{d}} |X(t)- y|^{2}\exp{\Big(-\frac{1}{2}| X(t)  - y|^{2}\Big)}\mathcal{L}_{X(t)}(\d y)
    \\  & 
    \leq \frac{1}{2}|X(t)  -q|^{2} + \int_{\mathbb{R}^{d}}\frac{1}{2}|X(t)- y|^{2}\mathcal{L}_{X(t)}(\d y).
\end{align*}
Using analogous arguments employed to obtain (\ref{equation_rcbo_5.7}), we get 
\begin{align*}
    \frac{d}{dt}\E|X(t) -q |^{2} &\leq c_5 e^{-\eta_1 t/2} +  \lambda(t)\E|X(t)  -q|^{2} + \lambda(t)\E\int_{\mathbb{R}^{d}}|X(t)- y|^{2}\mathcal{L}_{X(t)}(\d y)
    \\   & \leq c_5 e^{-\eta_1 t/2} +  \lambda(t)\E|X(t)  -q|^{2} + 2\lambda(t) c_4 e^{-\eta_1 t}, \numberthis
    \label{equation_rcbo_5.14}
\end{align*}
where $ c_5>0$ is a constant independent of $t$.  Note that $\E|X(t) -q|^{2}$ is bounded uniformly in $t$. Requiring for the function $\lambda(t)$ to tend to zero sufficiently fast as $t \to \infty$, we get that $\E|X(t) -q|^{2}$ converges to a constant as $t \rightarrow \infty$. Therefore, there is an $x^* \in \bar{G}$ such that  $\E X(t) \rightarrow x^*$ as $t \rightarrow \infty$. The rest of the arguments to complete the proof are the same as used in the previous theorem.
\end{proof}

\section{Numerical tests}\label{sec:exp}

In this section, we first (Section~\ref{subsec:numschemes}) consider discrete approximations of the interacting particle systems and then perform several numerical experiments to demonstrate effectiveness of the CBO algorithms described in Sections~\ref{subsec_2.1} and \ref{subsec_2.2} for constrained optimization.

Since these CBO algorithms are probabilistic in nature and convergence is guaranteed only asymptotically (in $N$, $\alpha$ and $t$), we report the success rate of each experiment for a finite $N$, $\alpha$ and $t$. Each experiment is run $10^3$ times, and the success rate is simply the proportion of experiments which successfully find the true global minimizer. We declare an experiment successful if the final consensus is within $\varepsilon$ of the location of the true minimizer. For each experiment, we use the threshold $\varepsilon = 0.1$ unless otherwise specified. 

We note briefly that there are several parameters/controls for the CBO algorithms: (i) the inverse temperature, $\alpha$; (ii) the number of particles, $N$; (iii) strength of the drift and diffusion coefficients, $\beta$ and $\sigma$, respectively; (iv) the integration time, $t$; and (v) the step size, $h$, of a numerical scheme approximating the particle system. 
We do not attempt to propose a systematic way of choosing these parameters, but instead opt for sensible and practically effective choices in the following experiments. In general, $\alpha$ and $N$ should be chosen as large as possible. For $\alpha$, we run the risk of numerical instability if it is chosen too large. For $N$, the computational cost of the algorithm is at least $\mathcal{O} (N)$ (and $\mathcal{O}(N^2)$ in the repelling case), so there is a trade-off that must be balanced. As such, in Sections \ref{subsec:ackley}, \ref{subsec:towsend}, \ref{subsec:rastrigin}, and \ref{subsec:repelling} we explore success rates for varying values of $N$. Similar comments can be made about the number of steps of the numerical scheme. 

The strength of drift, $\beta$, and diffusion $\sigma$, coefficients can also depend on time. We experiment with both of them being constant in the experiments of Sections~\ref{subsec:ackley}, \ref{subsec:towsend}, and \ref{subsec:repelling} but also linearly increasing ($\beta$) and exponentially decreasing ($\sigma$) functions in Sections~\ref{subsec:rastrigin} and~\ref{subsec:pide}. 
In all the experiments for the penalty scheme (\ref{eq:scheme_penalty}) the optimal choice of the penalty $\epsilon =h$ was used.  


\subsection{Approximation of the interacting particle system}
\label{subsec:numschemes}

To implement optimization or sampling methods based on mean-filed SDEs, we need to be able to simulate the interacting particle system (\ref{eq:particle_system}). 
Consider a uniform discretization of the time-interval $[0,T]$, for a fixed $T > 0$, such that $t_{k+1} - t_{k} = h$, $k= 0,\dots,K-1$ and $T = Kh$. 

In the case of CBO models like the ones from Sections~\ref{subsec_2.1} and~\ref{subsec_2.2} we need numerical methods for (\ref{eq:particle_system}) converging in the almost sure sense. There are two types of mean-square/almost sure methods for reflected SDEs in the literature: penalty and projection (see e.g. \cite{PET95,SLO01}). Denote by $\Pi(x) $ the projection of $x$ on $\partial G$ if $x \notin \bar G$ else $\Pi(x) = x$.  Introduce the function (penalty)
$    \pi(x) = (x - \Pi(x)),$
which is half of the gradient of square of distance function of $x$ from $\partial G$. 
Let $Y^{i,N}_k$ be an approximation of (\ref{eq:particle_system}).             
The penalty and projection schemes for reflected SDEs adapted to the interacting particle system (\ref{eq:particle_system}) take the form 
\begin{itemize}
    \item[\textbf{(i)}] \textbf{Penalty scheme : }   
    \begin{align}
        Y^{i,N}_{k+1} &= Y^{i, N}_{k} + b(t_{k}, Y^{i, N}_{k}, \hat{\mu}_{\textbf{Y}^{N}_{k}})h+ \sigma(t_{k}, Y^{i, N}_{k}), \hat{\mu}_{\textbf{Y}^{N}_{k}})\Delta W^{i}(t_{k}) - \frac{h}{\epsilon}  \pi(Y^{i, N}_{k}),    \,\,  
        i = 1,\dots,N,  \label{eq:scheme_penalty}
    \end{align}
    where $\Delta W^{i}(t_{k}) = W(t_{k+1}) - W(t_{k})$ are the Wiener increments. The optimal choice of the penalty strength $\epsilon$ is $h$ \cite{SLO01}.
    \item[\textbf{(ii)}] \textbf{Projection scheme : }
    \begin{align}
        \bar{Y}^{i, N}_{k+1} &= Y^{i, N}_{k} + b(t_{k}, Y^{i, N}_{k}, \hat{\mu}_{\textbf{Y}^{N}_{k}})h + \sigma(t_{k}, Y^{i, N}_{k},\hat{\mu}_{\textbf{Y}^{N}_{k}})\Delta W^{i}(t_{k}) \nonumber \\ 
        Y^{i, N}_{k+1} &= \Pi(\bar{Y}^{i,N}_{k+1}),\;\;\;\; i=1,\dots,N.   \label{eq:scheme_projection}
    \end{align}
\end{itemize}
There are rather limited numerical analysis results for mean-square convergence of these two methods. In the case of convex polyhedrons the mean-square order of convergence of both methods is $1/2$ up to a logarithmic correction and in the case of smooth convex domains mean-square order of $1/4$ has been proved, see e.g. \cite{PET95,SLO01}.    
Like in \cite{tretyakov_consensus-based_2023} one can potentially show that  mean-square convergence of these methods is uniform in number of particles $N$, considering this aspect is beyond the scope of this paper.

For sampling methods, it is sufficient to use weak-sense schemes for approximating  (\ref{eq:particle_system}). Weak first-order methods for reflected SDEs are available in \cite{GN96a,BGT04,leimkuhler2023simplerandom,milstein_stochastic_2004}.

\subsection{Ackley function}
\label{subsec:ackley}

We begin with the Ackley function, which is a widely used benchmark
function for global optimization problems \cite{back1996evolutionary}. The function is defined as:
\begin{equation}
      f(x, y) = -20 \exp\left(-0.2 \sqrt{\frac{1}{2} (x^2 + y^2)}\right) - \exp\left(\frac{1}{2} (\cos(2\pi x) + \cos(2\pi y))\right) + 20 + e.
\end{equation}
We translate the function so that the global minimum
is located at $(2, 2)$ instead of $(0, 0)$. The feasible region is the closed ball of radius 3 centred at the origin:
\begin{equation}
\bar G = \{(x, y) : x^2 + y^2 \leq 9\}.
\end{equation}
The function with constraints is illustrated in Figure~\ref{fig:ackley_plot}. We choose  drift $\beta = 1$ and diffusion $\sigma = 4$, time $t=1$, $\alpha=10^4$. Both the penalty scheme (\ref{eq:scheme_penalty}) and the projection scheme (\ref{eq:scheme_projection}) are tested and the success rates for various number of particles, $N$, and number of time steps, $1/h$ (in other words, choice of $h$ with the fixed $t=1$), are reported in Table~\ref{tab:ackley}.  We observe that both penalty and projection schemes converge reliably to the global minimizer as $N$ and the number of iterations increase. Both approaches exhibit broadly similar performance, achieving near 100\% success for modestly large $N$ and sufficiently many iterations. The projection method performs slightly better, which could be due to the fact that when the consensus is computed, all particles are guaranteed to lie within the feasible region.

\begin{figure}[ht]
    \centering
    \includegraphics[width=0.5\linewidth]{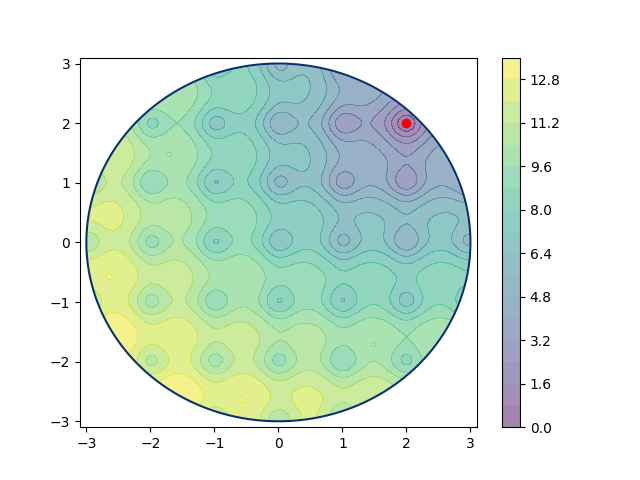}
    \caption{Translated Ackley function with minimum at $(2, 2)$ constrained to a closed ball of radius 3.}
    \label{fig:ackley_plot}
\end{figure}

\begin{table}[ht]
    \centering
    \caption{Translated Ackley function. Comparison of success rates for the penalty and projection schemes.}
    \label{tab:ackley}
    \begin{subtable}{0.45\textwidth}
        \centering
        \caption{Penalty scheme}
        \begin{tabular}{|c|r|r|r|r|}
            \hline
            \multirow{2}{*}{\textbf{$1/h$}} & \multicolumn{4}{c|}{\textbf{$N$}} \\
            \cline{2-5}
             & \textbf{10} & \textbf{20} & \textbf{50} & \textbf{100} \\
            \hline
            5   & 0.055 & 0.137 & 0.416 & 0.717 \\
            10  & 0.256 & 0.605 & 0.966 & 1.000 \\
            20  & 0.762 & 0.986 & 1.000 & 1.000 \\
            50  & 0.725 & 0.973 & 1.000 & 1.000 \\
            100 & 0.317 & 0.728 & 0.993 & 1.000 \\
            \hline
        \end{tabular}
    \end{subtable}
    \hspace{1cm}
    \begin{subtable}{0.45\textwidth}
        \centering
        \caption{Projection scheme}
        \begin{tabular}{|c|r|r|r|r|}
            \hline
            \multirow{2}{*}{\textbf{$1/h$}} & \multicolumn{4}{c|}{\textbf{$N$}} \\
            \cline{2-5}
             & \textbf{10} & \textbf{20} & \textbf{50} & \textbf{100} \\
            \hline
            5   & 0.137 & 0.370 & 0.809 & 0.979 \\
            10  & 0.550 & 0.926 & 1.000 & 1.000 \\
            20  & 0.883 & 0.997 & 1.000 & 1.000 \\
            50  & 0.730 & 0.978 & 1.000 & 1.000 \\
            100 & 0.307 & 0.717 & 0.993 & 1.000 \\
            \hline
        \end{tabular}
    \end{subtable}
\end{table}

\subsection{Non-convex function with heart-shaped constraint}
\label{subsec:towsend}

We now consider a non-convex function from \cite{townsend} defined as
\begin{equation}
f(x, y) = -\big[ \cos((x-0.1)y) \big]^2 - x \sin(3x+y)
\end{equation}
constrained to a heart-shaped region given by the inequality:
\begin{equation}
x^2 + y^2 \leq \Big[2\cos (t) + -\frac{1}{2}\cos(2t) - \frac{1}{4}\cos (3t) - \frac{1}{8}\cos(4t) \Big]^2 + 4 \sin^2(t),
\end{equation}
where $t = \mathrm{atan2}(x, y)$. The function and the feasible region are shown in Figure~\ref{fig:townsend_plot}. The level-set nature of the constraint means that we can project particles inside the domain in the following manner. First, if a particle lies outside the feasible region, the inward normal direction to the level set can be computed. Second, using this normal direction, we can solve - using Newton's method - for the distance along this normal we need to translate the particle such that it will lie on the boundary.
Note that the domain is non-convex.

\begin{figure}[ht]
    \centering
    \includegraphics[width=0.5\linewidth]{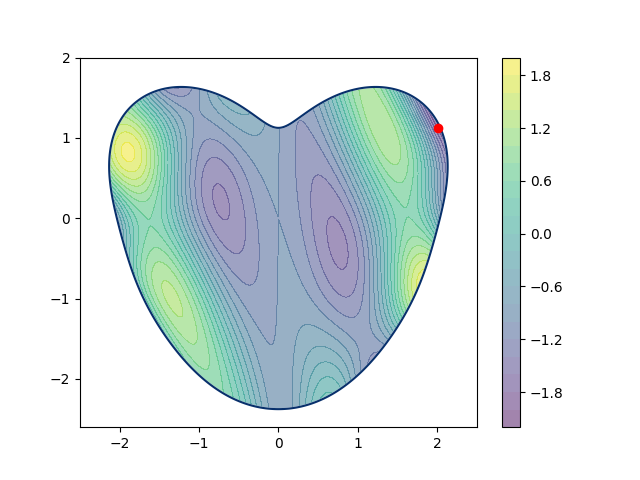}
    \caption{Non-convex function with heart-shaped constraint.}
    \label{fig:townsend_plot}
\end{figure}

Like before, we use drift $\beta = 1$, diffusion $\sigma = 4$ and $\alpha = 10^4$. Further, we use a fixed time step of $h = 1/20$. We report success rates using the projection scheme (\ref{eq:scheme_projection}) for various values of $N$ and number of time steps $K$ in Table~\ref{tab:townsend}.
The results in Table~\ref{tab:townsend} indicate that the CBO method reliably finds the global minimizer despite the non-convex nature of the domain. The observed success rates generally increase with $N$ and time (number of steps $K$), as would be expected.

\begin{table}[ht]
\caption{Success rates of the CBO method on the non-convex function with heart-shaped constraint.}
    \label{tab:townsend}
    \centering
    \begin{tabular}{|c|r|r|r|r|}
        \hline
        \multirow{2}{*}{$K$} & \multicolumn{4}{c|}{${N}$} \\
        \cline{2-5}
         & \textbf{10} & \textbf{20} & \textbf{50} & \textbf{100} \\
        \hline
        5   & 0.29 & 0.54 & 0.90 & 0.95 \\
        10  & 0.46 & 0.69 & 0.88 & 0.99 \\
        20  & 0.54 & 0.81 & 0.95 & 1.00 \\
        50  & 0.59 & 0.78 & 0.98 & 0.99 \\
        100 & 0.57 & 0.81 & 0.97 & 1.00 \\
        \hline
    \end{tabular}
    
\end{table}

\subsection{High-dimensional Rastrigin function}
\label{subsec:rastrigin}
We consider the Rastrigin function in various dimensions. It is defined as
\begin{equation}
f({x}) = 10d + \sum_{i=1}^{d}\left(x_i^2 - 10\cos(2\pi x_i)\right),
\end{equation}
constrained to the closed ball of radius 5 centred at the origin:
$
\bar G = \{{x} \in \mathbb{R}^{d} : |\mathbf{x}| \leq 5\}.
$
 Instead of constant drift and diffusion coefficients, $\beta$ and $\sigma$, we use 
$    \beta(t) = 10 t,$  $\sigma(t) = 10 e^{- t \log10 },$
so that the drift increases by an order of magnitude over the interval $[0, 1]$, while $\sigma$ decreases by an order of magnitude. We take $\alpha = 10^4$ and a step size of $h = 1/500$ is used.

Success rates using the projection scheme (\ref{eq:scheme_projection}) are given in Table~\ref{tab:rastrigin_dim} for various values of dimension, $d$, and number of particles, $N$. We also consider the cases where we have the number of iterations $K = 200$, $500$, and $1000$. 
The results in Table~\ref{tab:rastrigin_dim} highlight the scalability of the algorithm to higher-dimensional settings. While success rates decline with increasing dimensionality, they remain high for modestly large $N$. We also see an increase in success rates as $K$ increases, particularly between $K=200$ and $K=500$.

\begin{table}[ht]
    \centering
    \caption{Comparison of success rates for the Rastrigin function in various dimensions for number of iterations $K=200, 500$, and $1000$.}
    \label{tab:rastrigin_dim}
    \begin{subtable}{0.45\textwidth}
        \centering
        \caption{$K = 200$}
        \begin{tabular}{|c|r|r|r|r|}
            \hline
            \multirow{2}{*}{$d$} & \multicolumn{4}{c|}{\textbf{$N$}} \\
            \cline{2-5}
             & \textbf{10} & \textbf{20} & \textbf{50} & \textbf{100} \\
            \hline
            5   & 0.194 & 0.472 & 0.834 & 0.982 \\
            20  & 0.080 & 0.203 & 0.549 & 0.800 \\
            100 & 0.031 & 0.086 & 0.237 & 0.417 \\
            500 & 0.030 & 0.047 & 0.075 & 0.141 \\
            \hline
        \end{tabular}
    \end{subtable}
    \hspace{1cm}
    \begin{subtable}{0.45\textwidth}
        \centering
        \caption{$K = 500$}
        \begin{tabular}{|c|r|r|r|r|}
            \hline
            \multirow{2}{*}{$d$} & \multicolumn{4}{c|}{\textbf{$N$}} \\
            \cline{2-5}
             & \textbf{10} & \textbf{20} & \textbf{50} & \textbf{100} \\
            \hline
            5   & 0.251 & 0.533 & 0.899 & 0.993 \\
            20  & 0.159 & 0.459 & 0.841 & 0.948 \\
            100 & 0.095 & 0.275 & 0.750 & 0.950 \\
            500 & 0.059 & 0.174 & 0.492 & 0.819 \\
            \hline
        \end{tabular}
    \end{subtable}
    \par\vspace{0.5cm}
    \begin{subtable}{0.45\textwidth}
        \centering
        \caption{$K = 1000$}
        \begin{tabular}{|c|r|r|r|r|}
            \hline
            \multirow{2}{*}{$d$} & \multicolumn{4}{c|}{\textbf{$N$}} \\
            \cline{2-5}
             & \textbf{10} & \textbf{20} & \textbf{50} & \textbf{100} \\
            \hline
            5   & 0.263 & 0.523 & 0.903 & 0.991 \\
            20  & 0.167 & 0.456 & 0.825 & 0.951 \\
            100 & 0.077 & 0.321 & 0.779 & 0.958 \\
            500 & 0.060 & 0.181 & 0.506 & 0.845 \\
            \hline
        \end{tabular}
    \end{subtable}
\end{table}

\subsection{Testing repelling CBO on Rosenbrock function}
\label{subsec:repelling}

Let us now consider the Rosenbrock function in two dimensions, defined as
\begin{equation}
    f(x, y) = (1 - x)^2 + 100(y - x^2)^2.
\end{equation}
In addition we consider the feasible region
\begin{equation}
    \bar{G} = \{(x, y) \in \R^2: x^2 + y^2 \leq 2\},
\end{equation}
such that the minimizer of $f$, $(1, 1)$, lies on the boundary of $G$. The feasible region and the Rosenbrock function are displayed in Figure~\ref{fig:rosenbrock_plot}.

\begin{figure}[ht]
    \centering
    \includegraphics[width=0.5\linewidth]{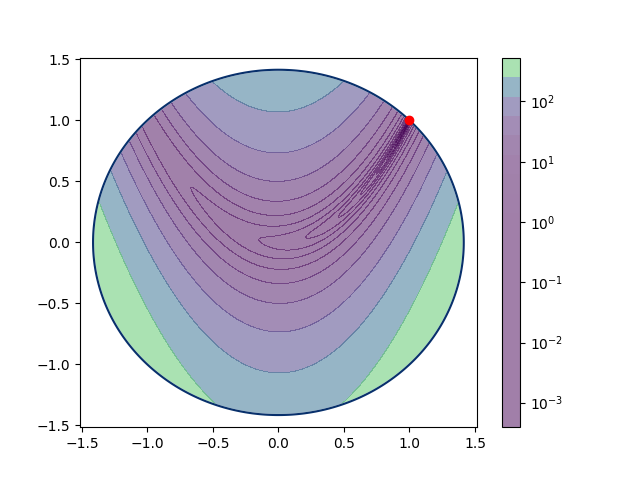}
    \caption{Rosenbrock function constrained to a closed ball of radius $\sqrt{2}$.}
    \label{fig:rosenbrock_plot}
\end{figure}

In this experiment, we compare the CBO method from Section~\ref{subsec_2.1} with the repelling CBO method from Section~\ref{subsec_2.2}. In Table~\ref{tab:repelling}, success rates for both methods are shown for different values of $N$ and $K$. We use the projection scheme (\ref{eq:scheme_projection}) with a fixed step size of $h=1/20$ for each experiment, as well as a drift parameter $\beta = 1$ and diffusion parameter $\sigma = 4$ and $\alpha = 10^4$.
The comparison in Table~\ref{tab:repelling} illustrates the better performance of the repelling CBO method over the standard CBO method, especially for smaller $N$ and $K$. The repelling method's ability to aid with exploration seems particularly relevant in this example, since it allows particles to continue exploring even when they are in a good location relative to other particles.

We remark that although the repelling method outperforms the standard CBO method, this is at the cost of additional computational complexity. Since pairwise interactions are computed, the repelling method is of $\mathcal{O}(N^2)$ complexity. However, it still requires only $N$ objective function evaluations per iteration. Thus, we can achieve better performance without any additional function evaluations which may be relevant when, for example, the objective function is expensive to evaluate. There are ways to overcome the $\mathcal{O}(N^2)$ complexity via locality sensitive hashing or efficient nearest neighbour search, see \cite{frenkel2023understanding} for example.

\begin{table}[ht]
    \centering
    \caption{Constrained Rosenbrock function. Comparison of success rates between standard and repelling constrained CBO}
    \label{tab:repelling}
    \begin{subtable}{0.45\textwidth}
        \centering
        \caption{Success rates for the standard CBO}
        \begin{tabular}{|c|r|r|r|r|}
            \hline
            \multirow{2}{*}{$K$} & \multicolumn{4}{c|}{${N}$} \\
            \cline{2-5}
             & \textbf{10} & \textbf{20} & \textbf{50} & \textbf{100} \\
            \hline
            5   & 0.075 & 0.106 & 0.192 & 0.332 \\
            10  & 0.084 & 0.135 & 0.281 & 0.525 \\
            20  & 0.094 & 0.169 & 0.422 & 0.766 \\
            50  & 0.121 & 0.249 & 0.702 & 0.982 \\
            100 & 0.133 & 0.315 & 0.892 & 0.999 \\
            \hline
        \end{tabular}
    \end{subtable}
    \hspace{1cm}
    \begin{subtable}{0.45\textwidth}
        \centering
        \caption{Success rates for the repelling CBO}
        \begin{tabular}{|c|r|r|r|r|}
            \hline
            \multirow{2}{*}{$K$} & \multicolumn{4}{c|}{${N}$} \\
            \cline{2-5}
             & \textbf{10} & \textbf{20} & \textbf{50} & \textbf{100} \\
            \hline
            5   & 0.182 & 0.340 & 0.665 & 0.863 \\
            10  & 0.214 & 0.384 & 0.669 & 0.834 \\
            20  & 0.234 & 0.433 & 0.765 & 0.915 \\
            50  & 0.263 & 0.538 & 0.929 & 0.999 \\
            100 & 0.287 & 0.595 & 0.979 & 1.000 \\
            \hline
        \end{tabular}
    \end{subtable}
\end{table}

\subsection{Inverse problem}
\label{subsec:pide}
In this section, we test the CBO model from Section~\ref{subsec_2.1} on an inverse problem.
Consider the Cauchy problem for the parabolic partial integral differential equation (PIDE): 
\begin{align}
    \label{eq:merton_pide}
    \pd{u}{t} &+ \frac{1}{2}\sigma^2 x^2 \pd{^2 u}{x^2} - b x \pd{u}{x} + \frac{1}{\sqrt{2\pi \gamma ^ 2}} \int_\R \big[ u (t, xe^y) - u(t, x) \big] \exp\Big(- \frac{(y - m)^2}{2 \gamma^2} \Big) \d y = 0,
\end{align}
with the terminal condition
\begin{align}
    u(T, x) = \max\{x - 1, 0\} 
\end{align}
and $b \vcentcolon= e^{m + \frac{1}{2}\gamma^2} - 1$.

The inverse problem is formulated within Tikhonov's regularization setup as follows. Given noisy observations $\hat u_{i,j}$ of the solution $u(t_i, x_j)$ to the PIDE problem (\ref{eq:merton_pide}) at some $t_i$, $x_j$, we aim to find estimates $\hat\theta:=(\hat \sigma, \hat m, \hat \gamma)$ of the parameters $\theta = (\sigma, m, \gamma)$ with the constraints $\hat \sigma \in [0, 1]$, $\hat m \in [-1, 1]$, and $\hat \gamma \in [0, 1]$. 
To achieve this aim, we construct the loss function 
\begin{align}
 \text{Loss}(\theta) =   \sum\limits_{i = 1}^{10}\sum_{j=1}^{5} \big| u(t_i, x_j; \theta) -   \hat u_{i,j}\big|^{2}+\lambda |\theta|_2, \label{eq:loss}
\end{align}
where $\lambda>0$ is a small regularization parameter. 
The data is synthetically generated using the following representation for the forward map \cite{merton1976option}:
\begin{align*}
u(t, x; \theta) &= \sum_{j=0}^\infty \frac{(\lambda \tau)^j}{j!} e^{-\lambda \tau} \Bigg\{ 
    x e^{-\lambda b \tau + j (m + \gamma^2 / 2)} \Phi \bigg( 
        \frac{\ln (x) + \big(  \frac{\sigma^2}{2} - \lambda b \big) \tau + j (m + \gamma^2)}{\sqrt{\sigma^2 \tau + j \gamma^2}} 
    \bigg) \\
&\quad -  e^{-r \tau} \Phi \bigg(
        \frac{\ln (x) + \big(  - \frac{\sigma^2}{2} - \lambda b \big) \tau + j m}{\sqrt{\sigma^2 \tau + j \gamma^2}}
    \bigg)
\Bigg\}, \,\, \tau = T - t,
\end{align*}
followed by adding a small amount of relative (observational) noise
\begin{align}
    \hat u_{ij} = u(t_i, x_j; \theta) + \epsilon_{ij}, \quad
    \epsilon_{ij} \sim \mathcal{N}\big(0, 10^{-3} u(t_i, x_j; \theta)\big).
\end{align}
Here $\Phi (\cdot )$ is the cdf for the standard normal distribution. 
Note that when the forward map is not available in an analytical form, linear PIDE problems can be solved using the Monte Carlo technique together with approximating the corresponding stochastic characteristics by a suitable numerical scheme (see e.g. \cite{George} and references therein).

For the experiments we take $T = 3$ and use the uniformly spaced grid  $t_i = 0.3 (i - 1)$, $i = 1, \ldots, 10$, and $x_j = 0.8 + 0.1(j - 1)$, $j = 1, \ldots, 5$.
As the ground truth, we chose $ \sigma = 0.1 $, $m = -0.2$, $\gamma = 0.3$, which are the true values to be recovered.
We use $N = 400$ particles and the projection scheme (\ref{eq:scheme_projection}) with step-size $h=0.01$ together with the drift and diffusion coefficients described in Section~\ref{subsec:rastrigin}. We run the projection scheme for $K=100$ steps. Additionally, we use the tolerance $\varepsilon = 0.01$ to determine if an experiment is successful.
Using a value of $\alpha = 10^{14}$, we obtain a success rate of 0.990. We note that in this case, a value of $\alpha = 10^4$ is not sufficient and results in a success rate of 0.07. 

In Figure~\ref{fig:parameter_histograms}, we include the histograms for each parameter in the final consensus of each experiment. We see that each estimated parameter is distributed reasonably tightly around the true parameter value.
In a future work it is of interest to test the constrained CBO models from Section~\ref{subsec_2.1} and~\ref{subsec_2.2} on more complex inverse problems. 

\begin{figure}[ht]
    \centering
    
    \begin{subfigure}[b]{0.32\textwidth}
        \includegraphics[width=\textwidth]{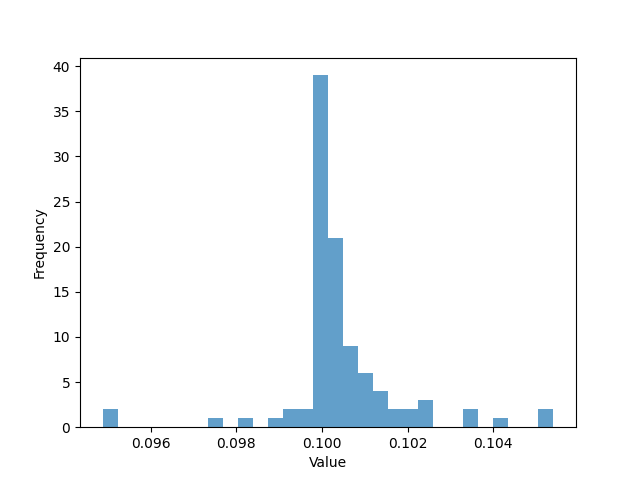}
        \caption{Histogram for $\sigma$}
    \end{subfigure}
    \hfill
    \begin{subfigure}[b]{0.32\textwidth}
        \includegraphics[width=\textwidth]{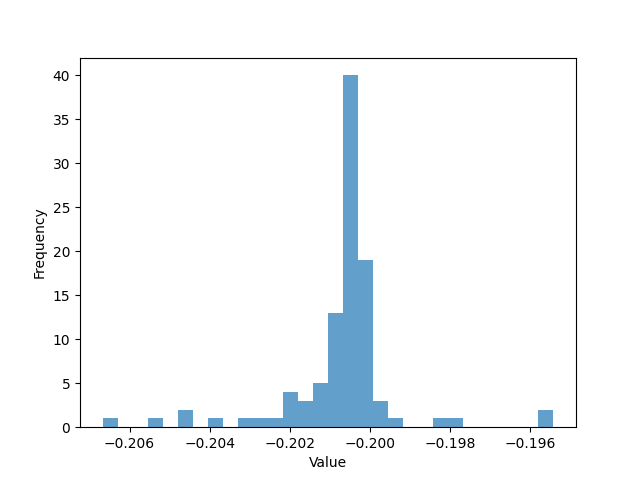}
        \caption{Histogram for $m$}
    \end{subfigure}
    \hfill
    \begin{subfigure}[b]{0.32\textwidth}
        \includegraphics[width=\textwidth]{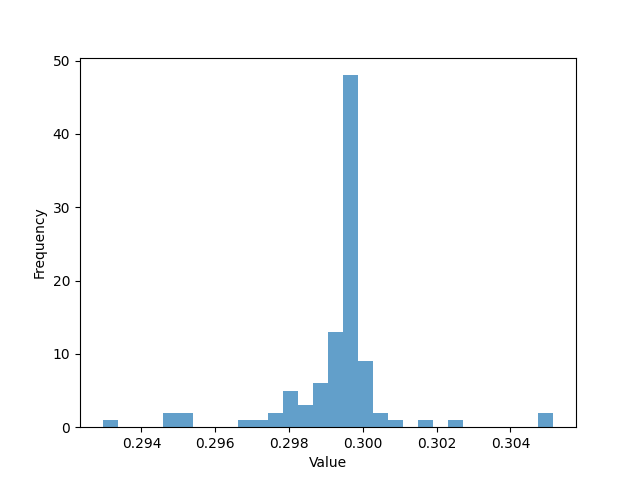}
        \caption{Histogram for $\gamma$}
    \end{subfigure}
    \caption{The inverse problem. Histograms of recovered parameters across the 1000 experiments.}
    \label{fig:parameter_histograms}
\end{figure}

\section*{Acknowledgments}
The authors are grateful to Professor Alain-Sol Sznitman and Dr Axel Ringh for useful discussions. AS and MVT were supported by the Engineering and Physical Sciences Research Council [grant number EP/X022617/1]. AS was also supported by the Wallenberg AI, Autonomous Systems and 
Software Program (WASP) funded by the Knut and Alice Wallenberg Foundation. AS and MVT thank the Isaac Newton Institute for Mathematical Sciences (Cambridge, UK), funded by EPSRC grant EP/Z000580/1, for support and hospitality during the programme ``Stochastic systems for anomalous diffusion'', where a part of work on this paper was undertaken. 

For the purpose of open access, the authors  applied a 
Creative Commons Attribution (CC-BY) license to any Author Accepted Manuscript version arising.

\bibliographystyle{plain}
\bibliography{ref}


\end{document}